\documentclass[11pt,reqno]{amsart}
\usepackage{amssymb}
\usepackage{amsmath,amsthm}
\usepackage{amsfonts}
\usepackage{leftidx}
\usepackage{mathrsfs}
\usepackage{enumerate}  
\usepackage{abstract}
\usepackage{stmaryrd}
\usepackage{ulem}
\usepackage{slashed}
 \usepackage{float}
\usepackage{graphicx}
 \usepackage{hyperref}
 
 \usepackage[top=1in, bottom=1in, left=0.8in, right=0.8in]{geometry}

\def\R{\mathbb{R}}
\def\Z{\mathbb{Z}}

\def\d{|\nabla|}

\def\p{\partial}

\def\vo{\vspace{1\baselineskip}}

\def\be{\begin{equation}}
\def\ee{\end{equation}}

\newtheorem{theorem}{Theorem}[section]

\newtheorem{lemma}{Lemma}[section]

\theoremstyle{definition}

\theoremstyle{remark}
\newtheorem{remark}{Remark}[section]

\numberwithin{equation}{section}
\begin{document}
\title[   2D Ericksen Leslie   ]{ Global solution of  2D hyperbolic liquid crystal system for small initial data }
 
\author{Xuecheng Wang}
\address{YMSC, Tsinghua University \& BIMSA\\ Beijing, China 100084}
\email{xuecheng@tsinghua.edu.cn}

\maketitle 

\begin{abstract}

We prove the global stability of small perturbation  near the constant equilibrium  for  the two dimensional simplified Ericksen-Leslie’s hyperbolic system
for incompressible liquid crystal model, where the direction function of  liquid crystal molecules satisfies a wave map
equation with an acoustical metric. This improves the almost global existence result by Huang-Jiang-Zhao \cite{HJZ}. As  byproducts, we obtain the sharp (same as the linear solution) $L^\infty_x$-decay estimates for both the heat part and the wave part. Moreover the nonlinear wave part scatters to a linear solution as time goes to infinity. 

This paper's main contribution is the discovery of a novel null structure within the velocity equation's wave-type quadratic self-interaction. This structure compensates the insufficient decay rate in 2D, which previously hindered the establishment of global regularity for small data.
\end{abstract}

\section{Introduction}
This paper is devoted to studying the small data global regularity problem for the  2D simplified Ericksen-Leslie's hyperbolic liquid crystal model, which  read as follows, 
\be\label{2Dmaineqn}
\left\{\begin{array}{l}
\p_t u +u\cdot \nabla u +\nabla p = \Delta u - \nabla\cdot\big(\nabla d \otimes \nabla d  \big), \quad \nabla\cdot u =0,\\ 
D_t^2 d -\Delta d = \big(- |D_t d |^2 + |\nabla_x d|^2 \big) d,
\end{array}\right.
\ee
where $D_t:=\p_t +u \cdot\nabla_x, d \in \mathbb{S}^1.$ For the general Ericksen-Leslie's system, see \cite{HJZ,Jiang-Luo-2018} and classic references \cite{Ericksen-1961-TSR,Ericksen-1987-RM,Ericksen-1990-ARMA,Leslie-1968-ARMA,Leslie-1979}. Conceptually, the simplified system arises by setting specific coefficients to zero in the general system.

If  we parametrize $d=(\cos\phi, \sin \phi )$, see Huang-Jiang-Zhao \cite{HJZ},   then the following system of equation holds, 
\be\label{mainsysnew}
\left\{\begin{array}{rl}
\p_t u - \Delta u &= - u\cdot\nabla u - \nabla p -\p_i( \nabla \phi \p_i\phi), \quad \nabla\cdot u =0\\
(\p_t^2- \Delta )\phi&= -u\cdot \nabla(\p_t\phi + u\cdot\nabla\phi) - \p_t (u\cdot \nabla \phi) \\
&=-2\p_i(u_i\p_t \phi) - \p_i(\p_t u_i  \phi)- \p_i(u_i u_j \p_j \phi) =:\mathcal{N}_\phi.  \\
\end{array}\right. 
\ee
 
For the physical background of the  system \eqref{mainsysnew}, we refer readers to \cite{HJZ} and references therein for more details. The system \eqref{mainsysnew} came to the author's attention when one of the authors of \cite{HJZ}, Prof. Ning Jiang, presented their small data almost global result in  ``Workshop on the recent progress of kinetic theory and related topics", Jan 15--19, 2024, Sanya, China.

\subsection{Main difficulty in $2D$ and  previous results}

Despite extensive research on nonlinear wave and heat equations, a comprehensive understanding of their quadratic interaction remains elusive.

 Thanks to the work of Cai-Wang \cite{CW} and Huang-Jiang-Luo-Zhao \cite{HJLZ1,HJLZ2}, the small data global regularity problem for the $3D$ case is now well understood. Unfortunately, comparing with the $t^{-1}$ decay rate in $3D$,   the optimal decay rate for the  nonlinear wave equation in $2D$ is only $t^{-1/2}$.  Due to the insufficient decay, the heat-type nonlinearities in \eqref{mainsysnew}, involving wave-type quadratic self-interaction, prevent closure of the energy estimate using a standard  $L^\infty_x-L^2_x$-type  bilinear estimate.  Consequently, establishing long-time existence for \eqref{mainsysnew} in $2D$ is a highly nontrivial problem; demonstrating global existence presents an even greater challenge.

  Interestingly,  Huang-Jiang-Zhao \cite{HJZ} successfully applied Alinhac's ghost weight method \cite{A01}, identifying a suitable unknown variable. While this method is well-established for wave equations, its application to the heat equation in this context is novel and yields an almost global existence result for small initial data.

The long time behavior of solutions beyond their almost global existence remained an open question. Specifically, do solutions exhibit finite-time blow-up, or do they remain bounded, and if so, what is their asymptotic behavior?

This paper aims to answer the above questions by demonstrating global existence and asymptotic scattering for the system \eqref{mainsysnew}. Our proof centers on uncovering a previously unrecognized null structure in the velocity equation, resulting from a cancellation between the pressure term $\nabla p $  and $ \p_i( \nabla \phi \p_i\phi)$. This null structure (see subsection \ref{null}) allows for effective control of nonlinear interactions (see subsection \ref{exploiting}), leading to the global existence  and scattering results. The inherent generality of this null structure implies its potential applicability across hyperbolic liquid crystal models. Notably, this structure also appears in the general Ericksen-Leslie's system.

\subsection{Main result of this paper}

The main result of this paper is stated as follows, 
\begin{theorem}\label{maintheorem}
Let $N_0=10^{5}$. There exists   absolute small constant $\epsilon_0$ such that if the initial data $(u_0, \phi_0, \phi_1)$ of the system \eqref{mainsysnew} satisfies the following smallness condition 
\be\label{initialdatacond}
\|u_0\|_{H^{N_0 +1}}+ \|U_0\|_{H^{N_0}} + \||x|\nabla_x  u_0\|_{L^2} +\sum_{|\alpha|\leq 1}  \||x|\nabla_x^\alpha  U_0\|_{L^2 }+\sum_{|\beta|\leq 11}  \| \nabla_x^\beta U_0\|_{L^1_x }+  \| |x| \nabla_x^\beta u_0\|_{L^1_x }\leq \epsilon_0, 
\ee
where $U_0:=\phi_1+ i \d \phi_0.$ Then the system \eqref{mainsysnew} has unique global solution. The nonlinear solution $\phi(t)$ scatters to a linear solution $\phi_\infty(t)$. Moreover, the following sharp decay estimates hold for $u$ and $\phi$, 
\be
 \sum_{|\alpha|\leq 3} \langle t \rangle^{1/2}\|\nabla_x^\alpha  \nabla_{t,x}  \phi\|_{L^\infty_x} + \langle t \rangle^{ }   \| \nabla_x^\alpha u(t)\|_{L^\infty} \lesssim  \epsilon_0.
\ee
\end{theorem}
\begin{remark}
The plausible goal of optimizing  $N_0$, which is less interesting, is not pursued here.
\end{remark}
\begin{remark}
Comparing with the result in \cite{HJZ}, we actually prove a stronger result about the energy of the heat part. Instead of allowing the energy of the heat part to grow at rate $\langle t\rangle^\delta$ (as in \cite{HJZ}),  we show that the energy of the heat component decays at a rate of $t^{-1/2+\delta}$, closely matching the decay rate of free heat flow  in $L^2$-space,   see \eqref{bootstrap}.  
\end{remark}

\subsection{Hidden null structure}\label{null}

This subsection presents a detailed analysis of the null structure inherent in the velocity equation \eqref{mainsysnew}. To facilitate this analysis, we first perform a reduction of the nonlinear term $\mathcal{N}_u$.

 Note that, by using the divergence free condition, we can solve the pressure as follows, 
\[
\Delta p = -\p_j(u_i\p_i u_j )- \p_i\p_j(\p_i \phi \p_j \phi ), \quad \nabla p = -\nabla \Delta^{-1}\p_j(u_i\p_i u_j )-  \nabla \Delta^{-1}\p_i\p_j(\p_i \phi \p_j \phi ). 
\]
Therefore, we can formulate the equation satisfied by velocity $u$ as follows, 
\be\label{velocityeqn}
\begin{split}
\p_t u - \Delta u & = \mathcal{N}_u:=  Q(u, u) +\widetilde{Q}(\phi, \phi ), \\
Q(u, u)&= -u\cdot \nabla u + \nabla \Delta^{-1}\p_j(u_i\p_i u_j )=-\p_i(u_i u)+  \nabla \Delta^{-1}\p_i \p_j(u_i u_j ), \\
\widetilde{Q}(\phi, \phi )& =  \nabla \Delta^{-1}\p_i\p_j(\p_i \phi \p_j \phi )-\p_i( \p_i\phi  \nabla \phi ) .
\end{split}
\ee

Note that that terms in  $\mathcal{N}_u$ always have one derivative outside, which contributes the smallness of output frequency. This observation  motivates us to formulate $\mathcal{N}_u$ as follows, 
\be\label{2024feb27eqn41}
\mathcal{N}_u=\d \widetilde{\mathcal{N}}_u, \quad \widetilde{\mathcal{N}}_u:=  -R_i(u_i u)-  R R_iR_j(u_i u_j )  - R R_i R_j (\p_i \phi \p_j \phi )-R_i( \p_i\phi  \nabla \phi ),
\ee
where $R_i=\p_{i}\d^{-1}, i\in\{1,2,3\},$ denote  the Riesz operator.

Furthermore, we observe a null   structure within the bilinear operator $\widetilde{Q}(\phi, \phi )$,  characterized by its symbol $\widetilde{q}(   \xi-\eta, \eta)$ detailed as follows, 
\be\label{symbolphi}
\begin{split}
\widetilde{q}(   \xi-\eta, \eta) &:= -i(\xi\cdot (\xi-\eta))|\eta| \big[ \frac{\xi (\xi\cdot \eta) }{|\xi|^2|\eta|} -  \frac{\eta}{|\eta|} \big].\\
\end{split}
\ee
The correspondence between the bilinear operator and its symbol is specified in equation \eqref{2025nov19eqn1}.

Very interestingly, the null structure presented in  $\widetilde{Q}(\phi, \phi )$ is not only helpful in the wave $\times$ wave $\rightarrow$ wave type interaction, which is more classic,  but also helpful in the wave $\times$ wave $\rightarrow$ heat type interaction, which is new. More precisely,  recall \eqref{symbolphi}, note that, 
\be\label{nullstrucdec}
\begin{split}
& \eta \cdot \xi = \frac{1}{2}\big[(|\eta|-|\eta-\xi|)(|\eta|+|\eta-\xi|) +|\xi|^2 ],\\ 
 \Longrightarrow \quad &  \widetilde{q}(  \xi-\eta,\eta)=q_{null}^{wwh}(\xi-\eta, \eta) q_{null}^{www}(\xi-\eta, \eta), \\
 & q_{null}^{wwh}(\xi-\eta, \eta):= i \frac{|\eta| }{2}\underbrace{\big[(|\eta|-|\eta-\xi|)(|\eta|+|\eta-\xi|) -|\xi|^2 \big]}_{\text{null structure for w$\times$w$\rightarrow$h}}, \\
 & q_{null}^{www}(\xi-\eta, \eta):= \underbrace{\big[ \frac{\xi (\xi\cdot \eta) }{|\xi|^2|\eta|} -  \frac{\eta}{|\eta|} \big].}_{\text{null structure for w$\times$w$\rightarrow$w}}
\end{split}
\ee

\subsection{Ideas of proof: exploiting the benefit of  the null structure}\label{exploiting}

 To better understand the notion of a     ``null structure for  w$\times$w$\rightarrow$h'',  we employ the following normal form transformation, 
\be\label{norform}
v:= u +\sum_{\mu, \nu \in\{+,-\}} A_{\mu, \nu}(U^{\mu}, U^{\nu}), \quad U:=(\p_t -i\d)\phi,
\ee
where $U^{+}:=U$, $U^{-}:=\bar{U}=(\p_t +i\d)\phi$, and the symbol $a_{\mu, \nu}(\xi-\eta, \eta)$ of the bilinear operator $A_{\mu, \nu}(U^{\mu}, U^{\nu})$ is given as follows, 
\be\label{2024feb18eqn1}
a_{\mu, \nu}(\xi-\eta, \eta)=  \frac{ -\mu\nu\widetilde{q}(  \xi-\eta,\eta)}{4\big(-|\xi|^2+i(\mu|\xi-\eta|+ \nu|\eta|)\big)} \frac{1}{|\xi-\eta||\eta|}. 
\ee

Note that, from \eqref{velocityeqn} and the above chosen symbol $a_{\mu, \nu}(\xi-\eta, \eta)$, which aims to cancel out the quadratic interaction of the wave components, we have 
\be\label{normalform}
\begin{split}
(\p_t - \Delta ) v &= Q(u, u) +\widetilde{Q}(\phi, \phi )+\sum_{\mu, \nu \in\{+,-\}} (-\Delta)\big( A_{\mu, \nu}(U^{\mu}, U^{\nu}) \big) + A_{\mu, \nu}(\p_t U^{\mu}, U^{\nu})+A_{\mu, \nu}( U^{\mu}, \p_tU^{\nu})\\
&= Q(u, u)  +  \sum_{\mu, \nu \in\{+,-\}}  A_{\mu, \nu}(\big( (\p_t +i \d) U\big)^{\mu}, U^{\nu})+A_{\mu, \nu}( U^{\mu}, \big( (\p_t +i \d) U\big)^{\nu})\\
&= Q(u, u)  +  \sum_{\mu, \nu \in\{+,-\}}  A_{\mu, \nu}(\big(\mathcal{N}_\phi \big)^{\mu}, U^{\nu})+A_{\mu, \nu}( U^{\mu}, \big( \mathcal{N}_\phi\big)^{\nu})=:\mathcal{N}_v(t).\\
\end{split}
\ee

Recall the detailed formula of $\widetilde{q}(  \xi-\eta,\eta)$ in \eqref{symbolphi} and the detailed formula of $a_{\mu, \nu}(\xi-\eta, \eta)$ in   \eqref{2024feb18eqn1}. The potential singularity of $a_{\mu, \nu}(\xi-\eta, \eta)$ is caused by the smallness of the factor $-|\xi|^2+i(\mu|\xi-\eta|+ \nu|\eta|)$, which vanishes only when $\xi=0$ and $\mu|\xi-\eta|+ \nu|\eta|=0$. Fortunately,  thanks to the ``null structure for w$\times$w$\rightarrow$h'', from  the detailed formula of $q_{null}^{wwh}(\xi-\eta, \eta)$ in \eqref{nullstrucdec}, it's apparent that $q_{null}^{wwh}(\xi-\eta, \eta)$ also vanishes in the same order  when $\xi=0$ and $\mu|\xi-\eta|+ \nu|\eta|=0$. Therefore, the symbol $a_{\mu, \nu}(\xi-\eta, \eta)$ of normal form transformation is non-singular.

 Since the nonlinearity of $v$ (see \eqref{normalform}) doesn't involve wave-type quadratic self-interaction,   the equation for $v$ is more amenable to analysis than the equation for $u$.  Naturally, from \eqref{norform}, we can also view that the velocity $u$ consists of the following two parts: 
\begin{enumerate}
\item[(i)] The heat part $v$, which is expected to exhibit asymptotic behavior consistent with a free heat solution. 
\item[(ii)] The wave-type quadratic self-interaction $ A_{\mu, \nu}(U^{\mu}, U^{\nu})$ constitutes a higher-order perturbation term that is decoupled from the heat flow. 
\end{enumerate}

Remarkably, the normal form transformation \eqref{norform} does not utilize the  w$\times$w$\rightarrow$w type null structure. Consequently, the bilinear form  $ A_{\mu, \nu}(\cdot, \cdot)$ retains its wave-type null structure. 

From \eqref{mainsysnew} and \eqref{norform}, after replacing $u$ by $v$ and $A_{\mu, \nu}(\cdot, \cdot)$ and neglecting the contribution of $v$,  we have the following approximation equation 
\be\label{approx}
(\text{Approximation equation}):\qquad \Box \phi = \sum_{\mu, \nu\in \{+,-\}}Q(A_{\mu, \nu}(U^{\mu}, U^{\nu}), \nabla_{t,x} \phi ),
\ee
where $Q(\cdot, \cdot)$ is some bilinear operator.  

Conceptually,  the lifespan of the solution for the approximation equation \eqref{approx} and the full nonlinear system \eqref{mainsysnew} should be identical. The main reason is that the velocity field decays faster than the wave component (see \eqref{Linfu} and \eqref{finalLinfyphi}), which means the approximation equation \eqref{approx} differs from the original system \eqref{mainsysnew} by a perturbative correction term that is integrable in time within the energy estimate.

Note that, the approximation equation is a 2D wave equation with  cubic nonlinearities. In the absence of a wave-type null structure, the 2D cubic wave equation is expected to exhibit finite-time blow-up, see e.g., \cite{Yin}.  The wave-type null structure of  $ A_{\mu, \nu}(\cdot, \cdot)$ 
  guarantees small data global solutions for the approximation equation \eqref{approx}, which is the primary reason for global solutions of the system \eqref{mainsysnew}.

\subsection{Methods of the proof}

Different from the physical space approach employed in \cite{HJZ}, we mainly analyze the system \eqref{mainsysnew} from the Fourier side. Moreover, we  use both  the strength of the vector field method and the Fourier method. Using the combination of these two methods is not new, e.g., it works successfully  in the seminal works of  Germain-Masmoudi-Shatah \cite{GerMasSha09,GerMasSha12,GerMasSha15} and Ionescu-Pusateri \cite{Ionescupusateri,Ionescupusateri2} for the study of water waves systems and Ionescu-Pausader  \cite{IP} for the study of Einstein-Klein-Gordon system,  see also Wang \cite{Wang} for the study of Einstein-Vlasov system.

For the sake of readers, we give a brief introduction on literature of these two methods.   The celebrated Klainerman vector field method was introduced in the seminal work   \cite{Kl85},   in which the symmetries    of the wave operator are exploited to prove decay estimate of the nonlinear wave equation, see also \cite{Kla82,Kla83,KM}.  A great  achievement of the vector field method lies in the monumental work, global stability of  the Minkowski spacetime,  by  Christodoulou-Klainerman \cite{ChrKl93}.  For the past fifteen years,  the Fourier method, which studies    nonlinear solutions on the Fourier side and   controls  the pull back of the nonlinear solution along the linear flow  over time, also plays an important role in the study of small data global regularity problem for nonlinear dispersive PDEs and NLW. Since the literature is too vast to survey, for the purpose of giving a sense for readers, we only mention the  spacetime resonance method and the $Z$-norm method.  The  spacetime resonance method was   introduced by Germain-Masmoudi-Shatah \cite{GerMasSha09} in the study of NLS. Now, it has very wide applications in the study of nonlinear dispersive equations, see e.g., \cite{GerMasSha12,GerMasSha14,IP1,IP2,IP3,IP} and  the nonlinear wave equations, see \cite{DenPus20,PusSha13}.   The  $Z$-norm method  was firstly introduced by Ionescu-Pausader in \cite{IP2}. This method is often used together with the spacetime resonance method. It   depends essentially on identifying the “correct”  $Z$-norm,  depending on the problem, to prove sharp  or almost sharp  decay estimates for the nonlinear solution.

\subsection{Main bootstrap assumption}
  
 To prove our main theorem \ref{maintheorem}, we employ the standard bootstrap argument.  Before stating our bootstrap assumption, we define the following normed space.
 \be\label{highphi}
 \begin{split}
  Z_{u}(t)&:=\sum_{k\in \Z} 2^{k+4k_+}  \|P_k u\|_{L^2},\\ 
 Z_{\phi}^{}(t)&:= \sum_{k\in \Z} 2^{k/3+10k_{+}}\big( \|\p_t\widehat{\phi}(t,\xi)\psi_k(\xi)\|_{L^\infty_\xi}+ 2^k \|\widehat{\phi}(t,\xi)\psi_k(\xi)\|_{L^\infty_\xi}\big). \\ 
 \end{split}
 \ee
For any $ \Gamma\in \{Id, S:= t\p_t +x\cdot \nabla_x, \Omega:=x_1\p_{x_2}-x_2\p_{x_1}\}, $   
 we define the   half wave and its profile as follows, 
\be\label{profphi}
\begin{split}
U^\Gamma(t) &:=(\p_t -i \d)\Gamma\phi, \quad 
 \Longrightarrow\quad \p_t \Gamma \phi  = \frac{U^\Gamma (t)+\overline{U^\Gamma(t)}}{2}, \quad \Gamma \phi = \sum_{\mu\in \{+,-\}}\frac{\mu}{2i\d} (U^\Gamma(t))^{\mu},  \\
 V^\Gamma(t)&:=e^{i t\d} U^\Gamma(t),\quad \Longrightarrow \p_t V^\Gamma(t)= e^{i t\d} \mathcal{N}^\Gamma_\phi, \quad U^\Gamma(t)= e^{-it\d} U^\Gamma(0) + \int_0^t e^{-i(t-s)\d} \mathcal{N}^\Gamma_{\phi}(s) d s.   \\
\end{split}
\ee
For simplicity in notation, we abbreviate $U^{Id}$ and $V^{Id}$   as $U$ and $V$ respectively.

With above preparations,  now we state the main bootstrap assumption   as follows, 
 \be\label{bootstrap}
 \begin{split}
\sup_{t\in [0, T]} &\langle t \rangle^{1/2-\delta }\|u(t)\|_{H^{N_0+1}}+ \langle t \rangle^{1/2-2\delta }\big(\sum_{\Gamma\in \{S, \Omega\}} \|\Gamma u(t)\|_{H^{1}} \big)  + \langle t\rangle^{}  Z_u(t) + \langle  t \rangle^{-\delta} \|U(t)\|_{H^{N_0}}\\
 &+ \langle t \rangle^{-2\delta }\big(\sum_{\Gamma\in \{S, \Omega\}} \|U^\Gamma(t)\|_{L^{2}} + \sum_{k\in \Z_{-}} 2^{-\alpha k} \| P_k U^{\Gamma}(t)\|_{L^2} \big) + Z_{\phi}(t) \lesssim \epsilon_1:=\epsilon_0^{3/4}, \quad \alpha:=1/10, \\
\end{split}
 \ee
 where, from the local theory, $T>0$ is some large constant. 

As a direct consequence of the above bootstrap assumption, we have the   sharp decay estimate for the velocity field $u$ as follows, 
\be\label{Linfu}
\sum_{0\leq |\alpha|\leq 3} \|\nabla_x^\alpha u(t)\|_{L^\infty}\lesssim \langle t \rangle^{-1}\epsilon_1. 
\ee
The rest of this paper is devoted to improve the upper bound in the bootstrap assumption \eqref{bootstrap}.

\subsection{Outline of this paper}

\begin{enumerate}
\item[$\bullet$] In section \ref{prel}, we introduce the notation used in this paper, a super-localized decay estimate, which plays an important role in later High $\times$ High type interaction. 

\item[$\bullet$]In section \ref{energy}, we do energy estimates for both the velocity part and the wave part.  
\item[$\bullet$] In section \ref{decay}, we estimate the $Z$-norms for the velocity part and the wave part, which give us the sharp decay estimates for nonlinear solutions.  
\end{enumerate}
\vo

 \noindent \textbf{Acknowledgment}\quad  The author acknowledges support from NSFC-12322110,12326602,12141102,  MOST-2020YFA0713003, and New Cornerstone Investigator Program 100001127.  The author gratefully acknowledges the anonymous referee’s comments and suggestions, which improved this paper’s presentation.
 \vo 

\noindent \textbf{Data/code availability statement:}\quad  The manuscript has no associated data or code. 

\vo 

\noindent \textbf{Conflict of interest:}\quad The author declares that he  has no conflict of interest.

 \section{Preliminary}\label{prel}
 \subsection{Notation}
 For any two numbers $A$ and $B$, we use  $A\lesssim B$, $A\approx B$,  and $A\ll B$ to denote  $A\leq C B$, $|A-B|\leq c A$, and $A\leq c B$ respectively, where $C$ is an absolute constant and $c$ is a sufficiently small absolute constant. We use $A\sim B$ to denote the case when $A\lesssim B$ and $B\lesssim A$.

For any $k\in \Z$, we use $k_{+}$ to denote $\max\{k,0\}$ and use $k_{-}$ to denote $\min\{k,0\}$. Moreover, for any $k\in \Z,$ we use  $P_{k}$, $P_{\leq k}$ and $P_{\geq k}$ to denote the projection operators  by the Fourier multipliers $\psi_{k}(\cdot),$ $\psi_{\leq k}(\cdot)$ and $\psi_{\geq k }(\cdot)$ respectively. We use $P_{[k_1,k_2]}$ to denotes $\sum_{k\in[k_1,k_2]}P_k$.  For convenience in notation, we also use  $f_{k}(x)$ to abbreviate $P_{k} f(x)$. We use both $\widehat{f}(\xi)$ and $\mathcal{F}[f](\xi)$ to denote the Fourier transform of $f$ and use $\mathcal{F}^{-1}[g](x)$ to denote the inverse Fourier transform of $g$.

 The symbol of the bilinear form $A(f,g)$ is defined as the function  $a(\xi-\eta, \eta)$, where the action of the operator is given by:
\be\label{2025nov19eqn1} 
A(f, g)(x) := \int_{\mathbb{R}^3} \int_{\mathbb{R}^3} e^{i x \cdot \xi} a(\xi - \eta, \eta) \widehat{f}(\xi - \eta) \widehat{g}(\eta) \, d\eta \, d\xi.
\ee 
 For any $m\in \Z_+$, the symbol $m(\underbrace{\xi-\eta,\eta-\sigma,\cdots, \kappa}_{\text{$m$-number of inputs}})$ of the $m$-linear operator $T_m(f_1,\cdots, f_m)$ can be defined similarly.

We define a class of symbol and its associated norms as follows,
\be\label{symbolclass}
\begin{split}
&\mathcal{S}^\infty(\R^4):=\{ m: m:\R^2\times\R^2
\rightarrow \mathbb{C} ,    \| m\|_{\mathcal{S}^\infty(\R^4)}:=\sum_{|\alpha|+|\beta|\leq 20}\|  \xi^\alpha \eta^\beta \nabla_\xi^\alpha \nabla_\eta^\beta m(\xi, \eta)\|_{L^\infty_{\xi, \eta}}< \infty\},\\ 
&\mathcal{S}^\infty(\R^6):=\{ m: m:\R^2\times\R^2\times\R^2
\rightarrow \mathbb{C} ,  \\ 
&\qquad\qquad  \| m\|_{\mathcal{S}^\infty(\R^6)}:=\sum_{|\alpha|+|\beta|+|\gamma|\leq 20}\|  \xi^\alpha \eta^\beta  \sigma^\gamma \nabla_\xi^\alpha \nabla_\eta^\beta\nabla_\sigma^\gamma m(\xi, \eta,\sigma)\|_{L^\infty_{\xi, \eta,\sigma}}< \infty\},\\ 
&\|m(\xi,\eta)\|_{\mathcal{S}^\infty_{k,k_1,k_2}}:= \| m(\xi, \eta)\psi_k(\xi)\psi_{k_1}(\xi-\eta)\psi_{k_2}(\eta)\|_{\mathcal{S}^\infty(\R^4)},\\
& \|m(\xi,\eta,\sigma)\|_{\mathcal{S}^\infty_{k,k_1,k_2,k_3}}:=\|m(\xi, \eta,\sigma)\psi_k(\xi)\psi_{k_1}(\xi-\eta)\psi_{k_2}(\eta-\sigma)\psi_{k_3}(\sigma)\|_{\mathcal{S}^\infty(\R^6)}.\\
 \end{split}
\ee

 For any chosen threshold $\bar{l}\in \Z_{-}, l\in [\bar{l},2]\cap \Z$, we define the cutoff function with the chosen threshold  $\varphi_{l;\bar{l}}(x):\R^3\rightarrow \R$  as follows, 
\be\label{cutoffthr}
\varphi_{l;\bar{l}}(x):= \left\{\begin{array}{ll}
\psi_{\leq \bar{l}}(x) & \textup{if}\, l=\bar{l},\\
\psi_{l}(x) & \textup{if}\, l\in (\bar{l}, 2),\\
\psi_{\geq 2}(x) & \textup{if}\, l= 2.\\
\end{array}\right. 
\ee

 We  define the following index sets, which correspond to   High $\times$ High type interaction, Low $\times$ High type interaction, and High $\times$ Low type interaction respectively, 
\be\label{2024march17eqn1}
\begin{split}
\chi_k^1&:=\{(k_1, k_2):k_1,k_2\in \Z, |k_1-k_2|\leq 10, k_1\geq k-10\},\\
\chi_k^2&:=\{(k_1, k_2):k_1,k_2\in \Z, |k_2-k|\leq 10, k_1\leq  k-10\},\\
\chi_k^3&:=\{(k_1, k_2):k_1,k_2\in \Z, |k_1-k|\leq 10, k_2\leq k-10\}.\\
\end{split}
\ee

 To study the   High$\times$High $\longrightarrow$ Low type interaction in the energy estimate, we exploit the following decay estimate with  the super-localized cutoff function. Comparing with the usual dyadic decomposition,   the width of the  annulus, which is the support of super-localized cutoff function,   is not comparable with the radius of annulus. The idea of using super-localized decay estimate to get around the summation issue of the   High$\times$High $\longrightarrow$ Low type interaction goes back to the work of Ionescu-Pausader \cite{IP} for the study of Einstein-Klein-Gordon system.

\begin{lemma}\label{decaylemma}
For any $t\in [2^{m-1}, 2^m], m\in \mathbb{Z}_+, k,\tilde{k}, n\in \mathbb{Z},  x\in \R^2, \mu \in\{+,-\}$, s.t., $\tilde{k}\leq k$, $n\in [2^{k-2}, 2^{k+2}]\cap \Z$,   $m\gg 1,$ $k\geq -m,$    we have 
\be\label{decayestimate}
\begin{split}
\big|\int_{\R^2}  e^{ i x\cdot \xi-i \mu t |\xi|  }     m(\xi) \widehat{f}(\xi)   \psi_{\tilde{k}}(|\xi|-n)    d \xi \big| & \lesssim    2^{-m/2 } \|m(\xi)\|_{\mathcal{S}^\infty_k}\big[ 2^{\tilde{k}+k/2}  \|\widehat{f}(\xi)\psi_{[\tilde{k}-2,\tilde{k}+2] }(|\xi|-n)\|_{L^\infty_\xi}  \\ 
& + 2^{-m/4+3k/4+\tilde{k}/2}\|\nabla_\xi \widehat{f}(\xi) \psi_{\tilde{k}}(|\xi|-n) \|_{L^2_\xi}  \big].\\
 \end{split}
\ee

\end{lemma}
\begin{proof}

$\bullet$\qquad If   $|x|\leq 2^{m-10} $.

Note that, for this case, we have $\big|\nabla_\xi(x\cdot\xi-\mu t|\xi|)\big|  \sim t \sim 2^m. $  To exploit the high oscillation in $\xi$, we do integration by parts in $\xi$ once. As a result, we have
\be\label{2024feb8eqn42}
\begin{split}
&\big|\int_{\R^2}  e^{ i x\cdot \xi-i \mu t |\xi|  }     m(\xi) \widehat{f}(\xi)  \psi_{\tilde{k}}(|\xi|-n)  d \xi \big| \\
& \lesssim \big|\int_{\R^2}  e^{ i x\cdot \xi-i \mu t |\xi|  } \nabla_\xi\cdot\big[\frac{(x|\xi|- \mu t \xi )}{|x|\xi|- \mu t \xi |^2 }  |\xi|  m(\xi) \widehat{f}(\xi)  \psi_{\tilde{k}}(|\xi|-n)\big]   d \xi \big| \\ 
&\lesssim 2^{-m}  \|m(\xi)\|_{\mathcal{S}^\infty_k} \big[2^k \|\widehat{f}(\xi)\psi_{[\tilde{k}-1, \tilde{k}+1]}(|\xi|-n) \|_{L^\infty_\xi} + 2^{(k+\tilde{k})/2}\|\nabla_\xi \widehat{f}(\xi) \psi_{\tilde{k}}(|\xi|-n) \|_{L^2_\xi}  \big]. 
\end{split}
\ee

$\bullet$\qquad If   $|x|\geq 2^{m-10} $.

Note that $\xi\times \nabla_\xi(x\cdot\xi-\mu t|\xi|)=0$ if and only if \mbox{${\xi}/{|\xi|}  = \mu{x}/{|x|} := \xi_{0} $}.  For this case, we do dyadic decomposition for the angle between $\xi$ and $\xi_0$ with the threshold $\bar{l}:=-m/2-k/2.$  From the volume of support of $\xi$, the following estimate holds for the small angle case, 
 \be\label{july27eqn2}
 \begin{split}
 &\Big| 
\int_{\R^2}    e^{ i x\cdot \xi-i \mu t |\xi|  }  
 \widehat{f}(\xi) m(\xi)  \psi_{\tilde{k}}(|\xi|-n) \psi_{\leq  \bar{l} }(\frac{\xi}{|\xi|}-\xi_0)  d \xi \Big|\\ 
 &\lesssim \|m(\xi)\|_{\mathcal{S}^\infty_k}  2^{k+\tilde{k}+\bar{l}} \|\widehat{f}(\xi) \psi_{\tilde{k}}(|\xi|-n) \|_{L^\infty_\xi}  \\
 & \lesssim 2^{-m/2+k/2+\tilde{k}} \|m(\xi)\|_{\mathcal{S}^\infty_k}   \|\widehat{f}(\xi) \psi_{\tilde{k}}(|\xi|-n)\|_{L^\infty_\xi}.    \\ 
  \end{split}
\ee
For the large angle case, we first do dyadic decomposition for the size of angle between $\xi$ and $\xi_0$ and then   do integration by parts in $\xi$  once. As a result, we have 
 \be\label{2024feb8eqn21}
 \begin{split}
 &\Big| 
\int_{\R^2}    e^{ i x\cdot \xi-i \mu t |\xi|  }  
 \widehat{f}(\xi) m(\xi)  \psi_{\tilde{k}}(|\xi|-n) \psi_{> \bar{l} }(\frac{\xi}{|\xi|}-\xi_0)  d \xi \Big|\\
 &\lesssim \sum_{l>\bar{l}} \big|\int_{\R^2}  e^{ i x\cdot \xi-i \mu t |\xi|  }  (\xi_2 \p_{\xi_1}-\xi_1\p_{\xi_2})\big[\frac{ m(\xi) \widehat{f}(\xi)   \psi_{\tilde{k}}(|\xi|-n)}{\xi_2 x_1-\xi_1 x_2}   \varphi_{l;\bar{l}}(\frac{\xi}{|\xi|}-\xi_0)  \big]   d \xi \big| \\
 &\lesssim \sum_{l>\bar{l}}  2^{-m-l}\|m(\xi)\|_{\mathcal{S}^\infty_k} \big[  2^{-k-l}2^{k+\tilde{k}+l} \|\widehat{f}(\xi)\psi_{k}(\xi)\|_{L^\infty_\xi}  + 2^{(k+\tilde{k}+l)/2} \|\nabla_\xi \widehat{f}(\xi)\psi_{k}(\xi)\|_{L^2_\xi}  \big] \\
  &\lesssim 2^{-m/2 } \|m(\xi)\|_{\mathcal{S}^\infty_k}\big[ 2^{\tilde{k}+k/2}  \|\widehat{f}(\xi)\psi_{k}(\xi)\|_{L^\infty_\xi} + 2^{-m/4+3k/4+\tilde{k}/2}\|\nabla_\xi \widehat{f}(\xi)\psi_{k}(\xi)\|_{L^2_\xi}  \big] .    \\ 
  \end{split}
\ee
To sum up, our desired estimate \eqref{decayestimate} holds from the above estimate and the obtained estimate  \eqref{2024feb8eqn42}.
\end{proof}

 By using the above linear decay estimate of wave equation, in the following Lemma, we show that the nonlinear wave part decays sharply over time under the bootstrap assumption \eqref{bootstrap}. 
\begin{lemma}\label{decayphilemma}
Under the bootstrap assumption \eqref{bootstrap}, the following $L^\infty_x$-decay estimate holds for the wave part, 
\be\label{finalLinfyphi}
\sum_{k\in \Z} \langle t \rangle^{1/2} 2^{-k/2+ 8k_{+}}\|P_k U\|_{L^\infty_{x}} \lesssim \epsilon_1.
\ee
\end{lemma}
\begin{proof}

Let $t\in [2^{m-1}, 2^m]\subset[0, T], m\gg1$. By using the volume of support of $\xi$, we first rule out the very low and relatively high frequency cases as follows,  
\[
\begin{split}
&\sum_{k\notin [-m/2, 2m/(N_0-20)] } 2^{-k/2+8k_+}\|\nabla_{t,x}P_k\phi\|_{L^\infty_{x}}\\
 &\lesssim  \sum_{k\notin [-m/2, 2m/(N_0-20)] } 2^{-k/2+3k_+}  \min\big\{2^{2k} \|\widehat{U}(t, \xi) \psi_k(\xi)\|_{L^\infty_{\xi}},  2^k \|\widehat{U}(t, \xi) \psi_k(\xi)\|_{L^2_{\xi}}  \big\} \\ 
& \lesssim \sum_{k\notin [-m/2, 2m/(N_0-20)]     } 2^{-k/2+8k_+}  \min\{2^{5k/3}, 2^{-(N_0-1)k+\delta m}\}\epsilon_1\lesssim 2^{-m/2}\epsilon_1. 
\end{split}
\]
Now, we focus on the case $k\in  [-m/2, 2m/(N_0-20)]\cap\Z$. Note that
\[
\big|\nabla_\xi \widehat{V}(t, \xi)\big|\lesssim t |\xi|^{-1} \big|\p_t\widehat{V}(t, \xi)\big| + |\xi|^{-1}\big[| \widehat{U^S}(t, \xi)|+| \widehat{U^\Omega}(t, \xi)|+| \widehat{U}(t, \xi)|\big].
\]
Recall \eqref{mainsysnew}. Note that, from \eqref{Linfu}, we have 
\[
\|\p_t\widehat{V}(t, \xi)\psi_k(\xi)\|_{L^2}\lesssim \|\nabla_{t,x}\phi\|_{H^3}\|u\|_{W^{3, \infty}}\lesssim \langle t \rangle^{-1+\delta}\epsilon_1^2.  
\]
Therefore, from the above two estimates and the bootstrap assumption \eqref{bootstrap}, we have
\be\label{2024feb16eqn31}
\sup_{k\in \Z} 2^{k-\alpha k_{-}} \big\|\nabla_\xi \widehat{V}(t, \xi)\psi_k(\xi)\big\|_{L^2}\lesssim \langle t \rangle^{\delta}\epsilon_1. 
\ee
For the rest of cases, we use the linear decay estimate \eqref{decayestimate} in Lemma \ref{decaylemma}. 
From  \eqref{2024feb16eqn31}, we have 
\be\label{2024feb16eqn1}
\begin{split}
&\| P_k U(t)\|_{L^\infty_x} \lesssim  2^{-m/2}\big(2^{k-10k_{+}} + 2^{-m/4+ (1/4+\alpha)k+\delta m} \big)\epsilon_1,\\
\Longrightarrow &\sum_{k\in [-m/2, 2m/(N_0-20)] } 2^{-k/2+8k_{+}}\|P_k U\|_{L^\infty_{x}} \lesssim   2^{-m/2}\epsilon_1. 
\end{split}
\ee
Hence finishing the proof of the desired estimate \eqref{finalLinfyphi}. 
\end{proof}

 \section{Energy estimate of the velocity part and the wave part }\label{energy}
 \subsection{Energy estimate of the velocity field}
\begin{lemma}\label{highenergyu}
Let $t\in [2^{m-1},2^m]\subset [0, T], m\in \Z_{+}$, s.t., $m\gg 1$. Under the bootstrap assumption \eqref{bootstrap}, the following improved energy estimate for the velocity field $u$ holds, 
\be\label{improvedenergyvel}
\| u(t)\|_{H^{N_0+1}}\lesssim 2^{-m/2+\delta m}\epsilon_0. 
\ee
\end{lemma}
\begin{proof}

 Recall the normal form transformation we did in \eqref{normalform}. We first estimate the energy of the normal form part. Thanks to the ``null structure for w$\times$w$\rightarrow$h'', the symbol $a_{\mu, \nu}(\xi-\eta, \eta)$ of normal form transformation is non-singular.  More precisely,   $\forall k, k_1, k_2\in \Z,$  we have
\be\label{symbolnormest}
\begin{split}
    \|a_{\mu, \nu}(\xi-\eta) \|_{\mathcal{S}^\infty_{k,k_1,k_2 }} +\| \frac{  q_{null}^{wwh}(  \xi-\eta,\eta)}{  -|\xi|^2+i(\mu|\xi-\eta|+ \nu|\eta|) }  \|_{\mathcal{S}^\infty_{k,k_1,k_2 }} \lesssim \left\{\begin{array}{ll}
   2^{-k_{+}} & \textup{if\,}(k_1, k_2)\in \chi_k^2\cup \chi_k^3,\\
   2^{-k_{1,-}} & \textup{if\,}(k_1, k_2)\in \chi_k^1,  \\
   \end{array}\right. \\
  \end{split}
\ee
where the sets of parameters $\chi_k^i, i\in\{1,2,3\}$, are defined in \eqref{2024march17eqn1}.    The validity of the above estimate follows directly from computing the norms defined in \eqref{symbolclass} for the symbol $a_{\mu, \nu}(\xi-\eta)$ ( see \eqref{2024feb18eqn1} and  \eqref{symbolphi}) and  the symbol $q_{null}^{wwh}(  \xi-\eta,\eta)$ (see \eqref{nullstrucdec}).

  For any $\mu, \nu\in\{+,-\},$ from the estimate of symbol in   \eqref{symbolnormest} and the $L^2-L^\infty$-type bilinear estimate, we have 
 \be\label{2024feb19eqn21}
 \begin{split}
&\sum_{\begin{subarray}{c}
|\alpha|\leq N_0+1\\
\end{subarray}} \|\nabla_x^\alpha \big(A_{\mu, \nu}(U^{\mu}, U^{\nu})\big)\|_{L^2} \lesssim \sum_{\begin{subarray}{c}
 |k_1-k_2|\leq 10\\
   k_1\geq k+20, |\alpha|\leq N_0+1\\
 \end{subarray}} 2^{ |\alpha|k-k_{1,-}}  \| P_k\big(A_{\mu, \nu}(P_{k_1}(U^{\mu}), P_{k_2}(U^{\mu})) \big)\|_{L^2} \\ 
&\quad  + \big(\sum_{k\in \Z, |\alpha|\leq N_0+1} 2^{2|\alpha|k-2k_+} \| P_k (U)\|_{L^2}^2 \| U\|_{L^\infty_x}^2  \big)^{1/2} \lesssim \big[\sum_{\begin{subarray}{c}
 |k_1-k_2|\leq 10\\
   k_1\geq k+20\\
 \end{subarray}}   \| P_k\big(A_{\mu, \nu}(P_{k_1}(U^{\mu}), P_{k_2}(U^{\mu})) \big)\|_{L^2}\big]\\ 
 &\quad +\big[ \sum_{\begin{subarray}{c}
 k_1\geq k+20,  k\in\Z\\ 
 1\leq |\alpha|\leq N_0+1\\ 
 \end{subarray} } 2^{ |\alpha|k-(N_0+2)k_{1,+}} \| U\|_{H^{N_0}}\|U\|_{W^{3, \infty}} \big]+     \| U\|_{H^{N_0}}\|U\|_{W^{1, \infty}}  \\ 
&\lesssim \sum_{\begin{subarray}{c}
 |k_1-k_2|\leq 10\\
   k_1\geq k+20\\
 \end{subarray}}   \| P_k\big(A_{\mu, \nu}(P_{k_1}(U^{\mu}), P_{k_2}(U^{\mu})) \big)\|_{L^2}  +  \| U\|_{H^{N_0}}\|U\|_{W^{3, \infty}}. 
 \end{split}
 \ee

 In the above estimate we single out the High $\times$ High type interaction  because of two reasons. Firstly,  the High $\times$ Low and Low $\times$ High type interactions are standard. Lastly and most importantly, the High $\times$ High type interaction is non-trivial in the sense that there is summability issue  with respect to the output frequency $k$ if without paying any price of decay rate. Note that this case only happens when $\alpha=0$, which explains why only $L^2$-norm appears in \eqref{2024feb19eqn21}.  

 To get around this summability issue, we first do super localization for two inputs and then use the super localized decay estimate \eqref{decayestimate} in Lemma \ref{decaylemma} and the orthogonality in $L^2$. Let $P_{k_1;k,n}$ denotes the Fourier multiplier operator with Fourier symbol $\psi_{k_1}(\xi)\psi_{k}(|\xi| - 2^k n )$. From  the estimate \eqref{2024feb16eqn31} and    the estimate of symbol in   \eqref{symbolnormest},   we have 
 \be\label{2024feb19eqn41}
 \begin{split}
&\| P_k\big(A_{\mu, \nu}(P_{k_1}(U^{\mu}), P_{k_2}(U^{\nu})) \big)\|_{L^2}   \lesssim \sum_{\begin{subarray}{c}
 n, m\in [2^{k_1-k-10}, 2^{k_1-k+10}]\cap \Z\\ 
 |n-m|\leq 2^{10}\\
 \end{subarray} } \| P_k\big(A_{\mu, \nu}(P_{k_1;k,n}(U^{\mu}), P_{k_2;k,m}(U^{\nu})) \big)\|_{L^2}\\
 &\lesssim \sum_{\begin{subarray}{c}
 n, m\in [2^{k_1-k-5}, 2^{k_1-k+5}]\cap \Z\\ 
 |n-m|\leq 2^{10}\\
 \end{subarray} } 2^{-m/2-k_{1,-}}\big[ 2^{k-k_{2,+}}\epsilon_1 + 2^{-m/4+3k_1/4+k/2} \|\nabla_\xi \widehat{V}(\xi) \psi_{k_2}(\xi)\psi_{k}(|\xi| - 2^k n ) \|_{L^2} \big] \\
 &\times \|P_{k_1;k,n}(U^{\mu})\|_{L^2} \lesssim 2^{-m/2-k_{1,-}}\min\{2^{-N_0 k_{1,+} +\delta m  }, 2^{k_1} \}(2^{k/2+k_1/2}+2^{-m/4+k/4+\delta m })\epsilon_1^2.
 \end{split}
 \ee
  After combining the obtained estimates \eqref{2024feb19eqn21}  and \eqref{2024feb19eqn41}, we have
  \be\label{2024feb19eqn42}
  \sum_{\begin{subarray}{c}
|\alpha|\leq N_0+1\\
\end{subarray}} \|\nabla_x^\alpha \big(A_{\mu, \nu}(U^{\mu}, U^{\nu})\big)\|_{L^2}\lesssim 2^{-m/2 +\delta m }\epsilon_1^2. 
  \ee

Now, we focus on the estimate of $H^{N_0}$-norm of $v$. Recall the equation satisfied by ``$v$'' in  \eqref{normalform}. Note that the following Duhamel's formula holds, 
 \be\label{duh}
 v(t) = e^{t \Delta} v(0) + \int_{0}^t e^{(t-s)\Delta}\mathcal{N}_v(s) d s. 
 \ee
 Note that, for any fixed $k\in \Z$, the following estimate holds for the frequency localized heat kernel,
 \be\label{heatkernelest}
 \big|\int_{\R^3} e^{i y\cdot \xi -(t-s)|\xi|^2} \psi_k(\xi) dy \big|\lesssim 2^{3k}(1+2^{2k}|t-s|)^{-100} (1+2^k|y|)^{-100}.
 \ee
 By using  the precise form of the heat kernel and H\"older inequality, we have
 \be\label{2024feb19eqn46}
 \begin{split}
\| v(t)\|_{H^{N_0+1}} & \lesssim 2^{-m/2}\epsilon_0 +\sum_{k\in \Z} \int_{0}^{t} (1+2^{2k}|t-s|)^{-10}\big[2^k \|P_k\mathcal{N}_v(s)\|_{L^1_x} +2^{(N_0+1)k }  \|P_k \mathcal{N}_v(s)\|_{L^2}\big]d s.  \\ 
 \end{split}
 \ee

Recall \eqref{normalform} and \eqref{velocityeqn}. We first estimate the $H^{N_0-1/2}_x$-norm of $\mathcal{N}_v(s)$.  From the estimate of symbol in \eqref{symbolnormest},  the  $L^2-L^\infty$-type bilinear estimate  and the decay estimates in  \eqref{Linfu} and \eqref{finalLinfyphi}, we have 
 \be\label{2024feb19eqn48}
 \begin{split}
 \|\mathcal{N}_v(s)\|_{H^{N_0-1/2}_x}&\lesssim \langle s\rangle^{-3/2+2\delta  }\epsilon_1^2+ \epsilon_1 \langle s\rangle^{-1/2}  \|\mathcal{N}_\phi(t)\|_{H^{N_0-3/2}_x} + \epsilon_1 \langle s\rangle^{\delta  }\|\mathcal{N}_\phi(t)\|_{L^\infty_x}\\
 & \lesssim \langle s\rangle^{-3/2+2\delta }\epsilon_1^2. \\
 \end{split}
 \ee

Recall \eqref{mainsysnew} and \eqref{2024feb27eqn41}. Note that $\mathcal{N}_\phi$   has one space derivative outside, which contributes the smallness of output frequency.  More precisely, we decompose $\mathcal{N}_\phi(s)$ as follows, 
\be\label{2024feb27eqn55}
\mathcal{N}_\phi(s)= -2\p_i(u_i\p_t \phi) - \p_i(\Delta u_i  \phi)- \p_i(u_i u_j \p_j \phi)  -  \p_i( \d \widetilde{\mathcal{N}}_u(s)  \phi(s))
\ee
Note that, after using the $L^2-L^\infty$-type bilinear estimate and using $Z_\phi(t)$-norm to control $\phi$,  we have
\be\label{2024feb27eqn51}
\begin{split}
\sum_{(k_1,k_2)\in \chi_k^2\cup \chi_k^3}& \| P_k(P_{k_1}(u_i) P_{k_2}(\p_t \phi))\|_{L^2_x} + \| P_k(P_{k_1}(\Delta u_i) P_{k_2}( \phi))\|_{L^2_x}  \\ 
& + \| P_k(P_{k_1}( \d \widetilde{\mathcal{N}}_u(s)  ) P_{k_2}( \phi))\|_{L^2_x} \lesssim 2^{k-3k_{+}}\langle s \rangle^{-1}\epsilon_1^2.
\end{split}
\ee
For the High $\times$ High type interaction, we use the  super-localized decay estimate for $\phi$-part. After employing the same strategy used in \eqref{2024feb19eqn41}, we have 
\be\label{2024feb27eqn52}
\begin{split}
\sum_{(k_1,k_2)\in \chi_k^1}& \| P_k(P_{k_1}(u_i) P_{k_2}(\p_t \phi))\|_{L^2_x} + \| P_k(P_{k_1}(\Delta u_i) P_{k_2}( \phi))\|_{L^2_x}  \\ 
& + \| P_k(P_{k_1}( \d \widetilde{\mathcal{N}}_u(s)  ) P_{k_2}( \phi))\|_{L^2_x} \lesssim 2^{k/4-3k_{+}}\langle s \rangle^{-1}\epsilon_1^2.
\end{split}
\ee
Lastly, from $L^2\rightarrow L^1$-type Sobolev embedding, we have
\be\label{2024feb27eqn53}
\|P_k(u_i u_j \p_j \phi)\|_{L^2_x}\lesssim 2^k\|P_k(u_i u_j \p_j \phi)\|_{L^1_x}\lesssim   2^{k-3k_{+}} \|u\|_{H^3}^2 \|U\|_{W^{3,\infty}_x}\lesssim 2^{k}\langle s \rangle^{-5/4}\epsilon_1^3.
\ee
To sum up, after combining   \eqref{2024feb27eqn55}--\eqref{2024feb27eqn53}, we have
\be\label{localizedphi}
\| P_k\big(   \mathcal{N}_\phi(s)\big)\|_{L^2_x}\lesssim 2^{5k/4-3k_{+}}\langle s \rangle^{-1 } \epsilon_1. 
\ee

For any fixed $k\in \Z$,  from  the obtained estimate \eqref{2024feb19eqn47},   $L^2-L^2$-type bilinear estimate,  the following estimate holds if we put $u$ in $L^2$ and $\phi$-part in $L^\infty_x$ for the $L^2_x$-estimate  of $\mathcal{N}_\phi$, 
 \be\label{2024feb19eqn47}
 \begin{split}
 &\|P_k \mathcal{N}_v(s)\|_{L^1}\lesssim \|u(s) \|_{L^2}\|\nabla u(s) \|_{L^2} +  \sum_{  (k_1,k_2)\in \cup_{i=1,2,3}\chi_k^i}     \| P_{k_2} U\|_{L^2} \| P_{k_1}\mathcal{N}_\phi\|_{L^2} \lesssim \langle s \rangle^{-1}\epsilon_1^2. 
 \end{split}
 \ee

 Moreover, 
After combining the above obtained   estimates \eqref{2024feb19eqn42}, \eqref{2024feb19eqn46}, \eqref{2024feb19eqn48}, and \eqref{2024feb19eqn47}, we have
\be\label{2024feb19eqn52}
\begin{split}
\| v(t)\|_{H^{N_0+1}} +\| u(t)\|_{H^{N_0+1}} &\lesssim 2^{-m/2+\delta m}\epsilon_0+\int_0^t (t-s)^{-1/2}\langle s \rangle^{-1}\epsilon_1^2+ (t-s)^{-3/4}\langle s \rangle^{-3/2+2\delta}\epsilon_1^2  d s \\ 
&  \lesssim 2^{-m/2+\delta m}\epsilon_0.   \\
\end{split}
\ee
Hence finishing the proof  of our desired estimate \eqref{improvedenergyvel}. 
\end{proof}

In the following Lemma, we control the energy of velocity with vector fields. 
\begin{lemma}\label{energyvelovec}
Let $t\in [2^{m-1},2^m]\subset [0, T], m\in \Z_{+}$, s.t., $m\gg 1$. Under the bootstrap assumption, the following improved energy estimate for the velocity field $u$ holds, 
\be\label{improvedenergyvelvec}
\sum_{\Gamma\in \{S, \Omega\}}\| \Gamma u(t)\|_{H^1}\lesssim 2^{-m/2+2\delta m}\epsilon_0. 
\ee
\end{lemma}
\begin{proof}
 Recall the normal form we did in  \eqref{norform}.  Again, we first estimate the energy of the normal form part.  Similar to the obtained estimate \eqref{2024feb19eqn41}, for any $\mu,\nu\in\{+, -\}, \Gamma\in \{S, \Omega\},$ after putting $U^\Gamma$ in $L^2$ and the other piece in $L^\infty$, from the sharp decay estimate in \eqref{finalLinfyphi} and the estimate of symbol in \eqref{symbolnormest}, we have
\[
  \|\Gamma \big(A_{\mu, \nu}(U^{\mu}(t), U^{\nu}(t))\big)\|_{H^1}\lesssim 2^{-m/2+2\delta m }\epsilon_1^2. 
\]

Now, we focus on the estimate of $v$-part. Note that, $S$ doesn't commute with the heat equation. As a result of direct computation,  we have 
 \[
 [\p_t -\Delta, \Omega]=0, \quad  [\p_t -\Delta, S]= [\p_t-\Delta, t\p_t+x\cdot \nabla_x] = \p_t-2\Delta 
 \]
 Therefore, 
 \[
 (\p_t -\Delta) S v  = S\mathcal{N}_v  + \mathcal{N}_v-\Delta v, \quad  (\p_t -\Delta) \Omega v = \Omega \mathcal{N}_v. 
 \]
 We rewrite the above two equations uniformly as follows,
 \[
 (\p_t -\Delta) \Gamma v  = \Gamma\mathcal{N}_v +c^1_{\Gamma}\mathcal{N}_v +c^2_{\Gamma}\Delta v, \quad c_S^1=1, c_S^2=-1, c_\Omega^1=c_{\Omega}^2=0.
 \]
 We also have the following Duhamel's formula for $\Gamma v(t)$, 
 \be\label{duhamelvelvec}
 \Gamma v(t)= e^{t \Delta} \Gamma v(0) + \int_0^t e^{(t-s)\Delta}\big(  \Gamma\mathcal{N}_v +c^1_{\Gamma}\mathcal{N}_v +c^2_{\Gamma}\Delta v \big) ds. 
 \ee

From the explicit formula of heat kernel and the obtained estimate \eqref{2024feb19eqn47}, we have
\be\label{2024feb20eqn1}
\begin{split}
\|\int_{0}^t e^{(t-s)\Delta} \mathcal{N}_v d s \|_{H^1_x}&\lesssim \sum_{k\in \Z} \int_0^t  (1+2^{2k}|t-s|)^{-10}  2^k\big(\|P_k\mathcal{N}_v(s)\|_{L^1_x}+\|P_k\mathcal{N}_v(s)\|_{L^2_x}\big)  ds \\ 
& \lesssim   \int_0^t (t-s)^{-1/2}  \langle s\rangle^{-1  } \epsilon_1^2   ds \lesssim 2^{-m/2+0.1\delta m }\epsilon_0.  \\ 
\end{split}
\ee
Again, from  the explicit formula of heat kernel, the following estimate holds from the obtained $H^{N_0}$-norm estimate of $v$ in  \eqref{2024feb19eqn52},  
\be\label{2024feb27eqn71}
\|\int_{0}^t e^{(t-s)\Delta} \Delta  v ds  \|_{H^1_x}\lesssim \int_0^{t-1} (t-s)^{-1}\| v (s)\|_{H^1_x} ds + \int_{t-1}^t \| v (s)\|_{H^3_x} ds \lesssim  2^{-m/2+1.1\delta m }\epsilon_0.   
\ee

Now, we focus on the  contribution of the main term $ \Gamma\mathcal{N}_v$.  Recall \eqref{normalform} and \eqref{velocityeqn}. Since $Q(u,u)$ has one space derivative outside, which contributes the smallness of output frequency, from $L^2-L^\infty$-type bilinear estimate, we have   
\[
 \|P_k\big(\Gamma \big(Q(u(s),u(s))\big)\big)\|_{L^2_x}\lesssim 2^{k-2k_{+}}\langle s \rangle^{-3/2+3\delta}\epsilon_1^2. 
\]
After doing dyadic decomposition for the output frequency and using  the pointwise estimate of the heat kernel,  the following estimate holds from the above obtained estimate, 
\be\label{2024feb21eqn1}
\begin{split}
\|\int_0^t e^{(t-s)\Delta} \big(\Gamma \big(Q(u(s),u(s))\big) d s \|_{H^1_x} &\lesssim \sum_{k\in \Z} \int_0^t (1+2^{2k}|t-s|)^{-10} 2^{k}\langle s \rangle^{-3/2+3\delta}\epsilon_1^2 ds \\
&\lesssim \int_0^t (t-s)^{-1/2} \langle s \rangle^{-3/2+3\delta}\epsilon_0 ds \lesssim 2^{-m/2}\epsilon_0. 
\end{split} 
\ee

Now, we consider the case   $\Gamma$ hits $ A_{\mu, \nu}(\big(\mathcal{N}_\phi \big)^{\mu}, U^{\nu})$ and $ A_{\mu, \nu}((U)^{\mu}, (\mathcal{N}_\phi)^{\nu})$. Due to symmetry, it would be sufficient to consider the case   $\Gamma$ hits $ A_{\mu, \nu}(\big(\mathcal{N}_\phi \big)^{\mu}, U^{\nu})$.   As a result of computation, we have
\be\label{2024feb27eqn61}
\begin{split}
\Gamma\big(A_{\mu, \nu}(\big(\mathcal{N}_\phi \big)^{\mu}, U^{\nu}) \big)&= A_{\mu, \nu}(\big( \Gamma \mathcal{N}_\phi \big)^{\mu}, U^{\nu})+ A_{\mu, \nu}(\big(\mathcal{N}_\phi \big)^{\mu}, (U^\Gamma)^{\nu})+ Comm_{\Gamma}(\phi) \\
Comm_{\Gamma}(\phi) &:=\Gamma\big(A_{\mu, \nu}(\big(\mathcal{N}_\phi \big)^{\mu}, U^{\nu}) \big)-  A_{\mu, \nu}(\big( \Gamma \mathcal{N}_\phi \big)^{\mu}, U^{\nu})- A_{\mu, \nu}(\big(\mathcal{N}_\phi \big)^{\mu}, (U^\Gamma)^{\nu})\\
\end{split}
\ee

 For the commutators, we use the same strategy of estimating $\mathcal{N}_v$ in \eqref{2024feb19eqn47} and \eqref{2024feb20eqn1}. As a result, we have
\be\label{2024feb27eqn74}
 \|\int_0^t e^{(t-s)\Delta} \big( Comm_{\Gamma}(\phi) \big) d s \|_{H^1_x} \lesssim 2^{-m/2+0.1\delta m}\epsilon_1^2.
 \ee
 Now, it would be sufficient to focus on the estimate of $A_{\mu, \nu}(\big(\Gamma \mathcal{N}_\phi \big)^{\mu}, U^{\nu})$ and $ A_{\mu, \nu}(\big(\mathcal{N}_\phi \big)^{\mu}, (U^\Gamma)^{\nu})$.  From the estimate of symbol in \eqref{symbolnormest}, the obtained estimate \eqref{localizedphi}, the  $L^2-L^\infty$-type bilinear estimate and the $L^\infty\rightarrow L^2$ type Sobolev embedding for the High $\times$ Low and Low $\times$ High type interactions,  the  $L^2-L^1$-type  Sobolev embedding  and the $L^\infty\rightarrow L^2$ type  bilinear estimate for the High $\times$ High   type interaction, we have
 \be\label{2024feb27eqn31}
\begin{split}
&  \sum_{(k_1, k_2)\in \chi_k^2 \cup \chi_k^3} \| P_k A_{\mu, \nu}(P_{k_1}\big(  \mathcal{N}_\phi(s) \big)^{\mu}, P_{k_2} (U^\Gamma )^{\nu}(s))    \|_{L^2_x}\lesssim 2^{9k/4-3k_{+}}\langle s \rangle^{-1+2\delta}\epsilon_1^2,\\
 &  \sum_{(k_1, k_2)\in \chi_k^1} \| P_k A_{\mu, \nu}(P_{k_1}\big(  \mathcal{N}_\phi(s) \big)^{\mu}, P_{k_2} (U^\Gamma  )^{\nu}(s))    \|_{L^2_x}\lesssim 2^{k-2k_{+}} \langle s \rangle^{-1+2\delta}\epsilon_1^2.\\
\end{split}
\ee

Similar to the obtained estimate \eqref{localizedphi}, we have
\[
\|P_k\big( \Gamma \mathcal{N}_\phi(s)\big) \|_{L^2_x}\lesssim    2^{k/2+ k_{+}/2}\langle s \rangle^{-1+2 \delta } \epsilon_1^2.  
\]
From the above estimate, the estimate of symbol in \eqref{symbolnormest},   the decay estimate \eqref{finalLinfyphi},    the  $L^2-L^\infty$-type bilinear estimate and the $L^\infty\rightarrow L^2$ type Sobolev embedding for the High $\times$ Low and Low $\times$ High type interactions,  the  $L^2-L^1$-type  Sobolev embedding  and the $L^\infty\rightarrow L^2$ type  bilinear estimate for the High $\times$ High   type interaction, we have
\be\label{2024feb27eqn11}
\begin{split}
 & \sum_{(k_1, k_2)\in \chi_k^2 \cup \chi_k^3} \| P_k A_{\mu, \nu}(P_{k_1}\big( \Gamma \mathcal{N}_\phi(s) \big)^{\mu}, P_{k_2} (U )^{\nu}(s))    \|_{L^2_x}\lesssim 2^{k/2-k_{+}/2}\min\{2^{k}\langle s \rangle^{-1+2\delta} , \langle s \rangle^{-3/2+2\delta} \}\epsilon_1^2,  \\
  & \sum_{(k_1, k_2)\in \chi_k^1} \| P_k A_{\mu, \nu}(P_{k_1}\big( \Gamma \mathcal{N}_\phi(s) \big)^{\mu}, P_{k_2} (U )^{\nu}(s))    \|_{L^2_x}\lesssim  2^{k-2k_{+}} \langle s \rangle^{-1+2\delta}\epsilon_1^2. \\
\end{split}
\ee

Recall the decomposition \eqref{2024feb27eqn61}. To sum up, after combining the above obtained estimates \eqref{2024feb27eqn74}--\eqref{2024feb27eqn11}, we have
\be\label{2024feb27eqn79}
\begin{split}
 \|\int_0^t e^{(t-s)\Delta}& \Gamma\big(A_{\mu, \nu}(\big(\mathcal{N}_\phi \big)^{\mu}, U^{\nu}) \big)\|_{H^1_x} 
  \lesssim 2^{-m/2+0.1\delta m}\epsilon_1^2+  \sum_{k\in \Z} \int_0^t (1+2^{2k}|t-s|)^{-10} 2^{k}\langle s \rangle^{-1+2\delta}\epsilon_1^2 d s  \\
 &\lesssim 2^{-m/2+0.1\delta m}\epsilon_1^2+  \int_0^t (t-s)^{-1/2}\langle s \rangle^{-1+2\delta}\epsilon_1^2 d s  \lesssim 2^{-m/2+2\delta m }\epsilon_1^2. 
 \end{split}
\ee
Therefore, our desired estimate \eqref{improvedenergyvelvec} holds  after combining the obtained estimates \eqref{duhamelvelvec}, \eqref{2024feb20eqn1}, \eqref{2024feb27eqn71}, \eqref{2024feb21eqn1}, and \eqref{2024feb27eqn79}. 
 \end{proof}

In the following Lemma, we control the energy   of the wave part. 
\begin{lemma}\label{energywavepart}
Let $t\in [2^{m-1},2^m]\subset [0, T], m\in \Z_{+}$, s.t., $m\gg 1$. Under the bootstrap assumption \eqref{bootstrap}, the following improved energy estimate for the wave part,  
\be\label{improvedenergywave}
2^{-\delta m}\|U(t)\|_{H^{N_0}} + 2^{-2\delta m}\big(\sum_{\Gamma\in \{S, \Omega\}}\|  U^\Gamma (t)\|_{L^2}+\sum_{k\in \Z_{-}} 2^{-\alpha k} \|P_k U^\Gamma(t)\|_{L^2}\big)\lesssim \epsilon_0. 
\ee
\end{lemma}
\begin{proof}

 Since the vector field $\Gamma\in \{S,\Omega\}$ plays the same role as $\nabla_x^\alpha, |\alpha|=N_0$,  the estimate of $ {H}^{N_0}$-norm of $U$ and the energy estimate of $U^{\Gamma}(t)$ follows in the same argument. We first give a detailed argument for   the $ {H}^{N_0}$-estimate  of $U$ by using a modified energy method and  then we define the corresponding modified energy for the vector field part.

  For   the $ {H}^{N_0}$-estimate  of $U$, we  define the following modified energy 
 \be\label{modifieden}
 \begin{split}
E_{modi}(t)&= \sum_{|\alpha|\leq N_0} \int_{\R^3} |\nabla_x^\alpha \nabla_x \phi |^2 +|\nabla_x^\alpha \p_t \phi |^2      - u_i u_j  \big( \p_i \nabla_x^\alpha \phi \p_j \nabla_x^\alpha  \phi  \big) + 2\p_t\nabla^\alpha_x\phi \nabla_x^\alpha u_i\p_i \phi  dx.  \\ 
 \end{split}
 \ee
 As a result of direct computations, we have
 \be\label{2024feb28eqn21}
  \begin{split}
 \frac{d  }{dt} E_{modi}(t) &= \sum_{|\alpha|\leq N_0} \int_{\R^3}  2\nabla_x^\alpha \p_t \phi \big(\p_t^2 - \Delta \big)\nabla_x^\alpha  \phi       -  2 u_i u_j  \big( \p_i \nabla_x^\alpha \p_t \phi \p_j \nabla_x^\alpha  \phi    \big)  + 2\p_t\nabla^\alpha_x\phi \p_t \nabla_x^\alpha u_i\p_i \phi  dx \\ 
 & + 2\p_t\nabla^\alpha_x\phi   \nabla_x^\alpha u_i\p_i \p_t \phi- 2 \p_t  u_i u_j  \big( \p_i \nabla_x^\alpha \phi \p_j \nabla_x^\alpha  \phi   \big)    + 2\p_t^2\nabla^\alpha_x\phi  \nabla_x^\alpha u_i\p_i \phi  dx.  \\ 
  \end{split}
 \ee
 Recall \eqref{mainsysnew}. Note that
 \be\label{feb8eqn1}
 \begin{split}
 (\p_t^2-\Delta)\nabla_x^\alpha \phi& = Semi_\alpha(\phi)+\sum_{i=1,2}Quas^i_{\alpha}(\phi), \quad  Semi_\alpha(\phi):=\nabla_x^\alpha \mathcal{N}_\phi - \sum_{i=1,2}Quas^i_\alpha(\phi)  , \\
Quas^1_\alpha(\phi)&:=   -  2u_i \p_t  \p_{x_i}\nabla_x^\alpha\phi  +u_i u_j\p_{x_i}\p_{x_j}\nabla_x^\alpha\phi,\quad Quas^2_\alpha(\phi):= -\p_t \nabla_x^\alpha u_i \p_i \phi. \\
\end{split}
 \ee

For the semilinear part, we use the $L^2-L^\infty$-type bilinear estimate. As a result, we have
\be\label{2024feb28eqn1}
\begin{split}
&\big| \int_{\R^3} \p_t\nabla_x^\alpha\phi Semi_{\alpha}(\phi ) d x\big|+ \big| \int_{\R^3}   \p_t  u_i u_j  \big( \p_i \nabla_x^\alpha \phi \p_j \nabla_x^\alpha  \phi \big)   d x\big| + \big|\int_{\R^3} \p_t\nabla^\alpha_x\phi   \nabla_x^\alpha u_i\p_i \p_t \phi d x \big|\\
&\lesssim \| U(t)\|_{H^{N_0}}\big[ \| U(t)\|_{H^{N_0}} \| u(t)\|_{W^{2, \infty}} +\| U(t)\|_{W^{2, \infty}} \| u(t)\|_{H^{N_0+1}}  \big]\lesssim \langle t \rangle^{-1+2\delta }\epsilon_1^3.
\end{split}
\ee

 Thanks to the symmetric structure and the divergence free condition of $u$, for the quasilinear part $Quas^1_\alpha(\phi)$, we have  
 \[
 \begin{split}
 &  \int_{\R^3} 2\p_t\nabla_x^\alpha\phi Quas^1_{\alpha}(\phi ) d x - \int_{\R^3}2 u_i u_j  \p_i   \p_t  \nabla_x^\alpha \phi \p_j \nabla_x^\alpha  \phi   dx  \\
&=  - \int_{\R^3}  2\p_t\nabla_x^\alpha \phi\big[ 2u_i \p_t  \p_{x_i}\nabla_x^\alpha\phi  +u_i u_j\p_{x_i}\p_{x_j} \nabla_x^\alpha \phi \big] d x - \int_{\R^3}2 u_i u_j  \p_i   \p_t \nabla_x^\alpha \phi \p_j \nabla_x^\alpha  \phi   dx   \\
&=  \int_{\R^3}  - 2\p_{x_i}\big( \p_t \nabla_x^\alpha \phi   \p_j \nabla_x^\alpha  \phi\big) u_i u_j d x  =    \int_{\R^3}   2\big( \p_t \nabla_x^\alpha \phi   \p_j  \nabla_x^\alpha \phi\big) u_i \p_{x_i} u_j d x. \\
\end{split}
 \]
 Therefore, by using the $L^2-L^\infty$-type bilinear estimate, we have
 \be\label{2024feb28eqn51}
  \begin{split}
 &\big|\int_{\R^3} 2\p_t\nabla_x^\alpha\phi Quas^1_{\alpha}(\phi ) d x - \int_{\R^3}2 u_i u_j  \p_i   \p_t  \nabla_x^\alpha \phi \p_j \nabla_x^\alpha  \phi   dx \big|\\
 &\lesssim \| U(t)\|_{H^{N_0}}\big[ \| U(t)\|_{H^{N_0}} \| u(t)\|_{W^{2, \infty}} +\| U(t)\|_{W^{2, \infty}} \| u(t)\|_{H^{N_0+1}}  \big]\\
  &\lesssim \langle t \rangle^{-1+2\delta }\epsilon_1^3. \\
 \end{split}
 \ee

Recall the definition of our modified energy in \eqref{modifieden}. The cubic correction term $ 2\p_t\nabla^\alpha_x\phi \nabla_x^\alpha u_i\p_i \phi$   in \eqref{modifieden}  aims to  cancel the quasilinear part $Quas^2_\alpha(\phi).$ More precisely, we have
\[
   \int_{\R^3} 2\p_t\nabla_x^\alpha\phi Quas^2_{\alpha}(\phi ) d x + \int_{\R^3}  2\p_t\nabla^\alpha_x\phi \p_t \nabla_x^\alpha u_i\p_i \phi    dx =0.
\]

Recall \eqref{2024feb28eqn21}. Now, it remains to estimate the last piece  $\p_t^2\nabla^\alpha_x\phi  \nabla_x^\alpha u_i\p_i \phi$. By using the equation satisfied by $\nabla_x^\alpha \phi$ in   \eqref{feb8eqn1}, we make the following decomposition,  
\[
\begin{split}
&\int_{\R^3} \p_t^2\nabla^\alpha_x\phi  \nabla_x^\alpha u_i\p_i \phi d x = I_1(t)+I_2(t), \\
&I_1(t):= \int_{\R^3} \Delta \nabla^\alpha_x\phi  \nabla_x^\alpha u_i\p_i \phi d x +  \int_{\R^3} \big[Semi_\alpha(\phi)+ Quasi^1_\alpha(\phi)\big]  \nabla_x^\alpha u_i\p_i \phi d x, \\
&I_2(t):= \int_{\R^3}   Quasi^2_\alpha(\phi)  \nabla_x^\alpha u_i\p_i \phi   d x =  -  \int_{\R^3} \p_t \nabla_x^\alpha u_i \p_i \phi \nabla_x^\alpha u_i\p_i \phi     d x . \\ 
\end{split}
\]
For $I_2(t)$, we use the equation satisfied by $u $ in \eqref{mainsysnew}. As a result, we have
\[
\begin{split}
I_2(t)&= I_2^1(t) +I_2^2(t), \quad I_2^1(t):=  -  \int_{\R^3} \Delta \nabla_x^\alpha u_i \p_i \phi \nabla_x^\alpha u_i\p_i \phi    d x, \quad
 I_2^2(t)=  -  \int_{\R^3}   \nabla_x^\alpha \mathcal{N}_u\cdot \nabla \phi \nabla_x^\alpha u_i\p_i \phi     d x. \\
\end{split}
\]

Since $u\in H^{N_0+1}$, which is one derivative better than $\nabla_{t,x}\phi \in H^{N_0}$. By using this fact,  it's clear that there is no losing derivative issue for $I_1(t),I_2^1(t)$, and $I_2^2(t).$ From the $L^2-L^\infty$-type bilinear estimate, we have 
\be\label{2024feb28eqn33}
\begin{split}
\big|\int_{\R^3} \p_t^2\nabla^\alpha_x\phi  \nabla_x^\alpha u_i\p_i \phi d x\big|&\lesssim \big|I_1(t)\big|+ \big|I_2^1(t)\big|+ \big|I_2^1(t)\big|\lesssim \|U(t)\|_{H^{N_0}}^2 \|U(t)\|_{W^{3,\infty}}^2 \\
&\quad + \|U(t)\|_{H^{N_0}}\|u(t)\|_{H^{N_0+1}} \|U(t)\|_{W^{3,\infty}}  + \|u(t)\|_{H^{N_0+1}}^2 \|U(t)\|_{W^{3,\infty}}\\
& \lesssim \langle t \rangle^{-1+2\delta}\epsilon_1^3. 
\end{split}
\ee
To sum up, from the above obtained  estimates \eqref{2024feb28eqn1}, \eqref{2024feb28eqn51},and \eqref{2024feb28eqn33},  we have
 \be\label{2024fen28eqn61}
 \begin{split}
 \| U(t)\|_{H^{N_0}}^2 & \lesssim E_{modi}(t) + \| U(t)\|_{H^{N_0}}^2  \| u(t)\|_{W^{2, \infty}}^2+  \| U(t)\|_{H^{N_0}} \| u(t)\|_{H^{N_0}}  \| U(t)\|_{W^{2, \infty}}\\
 &\lesssim \int_0^t \langle s \rangle^{-1+2\delta }\epsilon_1^3 d s + \epsilon_1^4\lesssim \langle t \rangle^{2\delta} \epsilon_0^2.\\
 \end{split} 
 \ee

Similar to what we did in the $H^{N_0}$-norm estimate, to avoid the losing derivative issue for the energy estimate of $U^\Gamma(t)$, we define the following modified energy, 
 \be\label{modifiedenvect}
 \begin{split}
E_{modi}^{vec}(t)&=  \sum_{\Gamma\in\{S, \Omega\}}\int_{\R^3}        |\p_t \Gamma \phi |^2 + |\nabla_x \Gamma \phi |^2  -\int_{\R^3} u_i u_j  \big(   \p_i \Gamma \phi \p_j \Gamma \phi  \big) + 2\p_t\Gamma\phi \Gamma u_i\p_i \phi dx.  \\ 
 \end{split}
 \ee
Again, due to the symmetric structure of $\mathcal{N}_\phi$ and the fact that $\Gamma u\in H^{1}$, which is one derivative better than $\nabla_{t,x}\Gamma\phi$.  With   minor modifications in the arguments for the estimate of $E_{modi}^{}(t)$, we have 
 \be\label{2024fen28eqn62}
 \begin{split}
 \big|\frac{d}{d t} E_{modi}^{vec}(t)\big| &\lesssim\sum_{\Gamma\in \{S, \Omega\}} \|U^{\Gamma}(t)\|_{L^{2}}^2 \|U(t)\|_{W^{3,\infty}}^2+ \|U^{\Gamma}(t)\|_{L^{2}}\|\Gamma u(t)\|_{H^{1}} \|U(t)\|_{W^{3,\infty}}\\ 
 & \quad  + \|\Gamma u(t)\|_{H^{1}}^2 \|U(t)\|_{W^{3,\infty}}\lesssim \langle t \rangle^{-1+4\delta}\epsilon_1^3. \\
 \Longrightarrow\quad     \| U^{\Gamma}(t)\|_{L^{2}}^2& \lesssim E_{modi}(t) +   \| U^{\Gamma}(t)\|_{L^{2}}^2  \| u(t)\|_{W^{2, \infty}}^2   +   \| U^{\Gamma}(t)\|_{L^{2}} \| \Gamma u(t)\|_{H^{1}}  \| U(t)\|_{W^{2, \infty}}\\
  &\lesssim \int_0^t \langle s \rangle^{-1+4\delta }\epsilon_1^3 d s + \epsilon_1^4\lesssim \langle t \rangle^{4\delta} \epsilon_0^2.\\
 \end{split} 
 \ee

Lastly, we consider the low frequency part (the size of the output frequency is less than one), for which there is no losing derivative issue. A more detailed explanation is provided as follows.  For the the High $\times$ High type interaction, there is no losing derivative issue since derivatives can be distributed between two inputs. Meanwhile,  for the High $\times$ Low type interaction and the Low $\times$ High type interaction, the frequency of the  High frequency part is comparable with the output frequency, which implies that   the size of the   High frequency input is less one. Consequently, there is no losing derivative issue. 

Recall \eqref{mainsysnew}. Since the nonlinearity $\mathcal{N}_\phi$ has a derivative outside, which contributions the smallness of $2^k$. Therefore, from the $L^2-L^\infty_x$ type bilinear estimate, for any $\Gamma\in \{S, \Omega\},$ we have
\be\label{2024marc3eqn1}
\begin{split}
 \sum_{k\in \Z_{-}} 2^{-\alpha k}\|P_k U^\Gamma(t)\|_{L^2}& \lesssim \epsilon_0 + \int_0^t \sum_{k\in \Z_{-}} 2^{-\alpha k}\big( \|P_k \Gamma \mathcal{N}_\phi(s)\|_{L^2}+\|P_k  \mathcal{N}_\phi(s)\|_{L^2}\big) d s \\
&\lesssim \epsilon_0+   \int_0^t \sum_{k\in \Z_{-}} 2^{(1/2-\alpha) k}\langle s\rangle^{-1+2\delta}\epsilon_1^2  ds \lesssim \langle t\rangle^{2\delta}\epsilon_1^2. 
\end{split}
\ee
 Hence finishing the proof of our desired estimate \eqref{improvedenergywave}.
 \end{proof}
\section{$Z$-norm estimate of the velocity part and the wave part}\label{decay}

In this section, we improve  the estimates of $Z_u(t)$ and $Z_\phi(t)$  within the time interval $[0,T].$ Hence finishing the bootstrap argument. Before   proceeding to the  the estimate of $Z_u(t)$ and $Z_\phi(t)$, we first prove a technical bilinear estimate, which will be used to study the bilinear interaction of wave part with angular localization on the Fourier side.  
 
Let $\mu, \nu\in \{+,-\}$, $m(\xi, \eta) \in  S^{\infty}, \phi\in C^\infty_0(\R^2)$, $l,k, k_1, k_2\in \mathbb{Z}$, $\alpha\in \R_+,$ $l\leq 0$,   and   $f,g \in L^2\cap L^1$.   We define the rescaled Schwartz function $\phi_{l}(x)=\phi(2^{-l}x)$ and  a bilinear form as follows, 
\be\label{2024feb19eqn1}
\begin{split}
T^l_{k;k_1, k_2}(f, g) &=   \int_{\R^2}  e^{it (|\xi|-\mu|\xi-\eta|-\nu |\eta|)} \psi_{k}(\xi) m(\xi, \eta) \big(\frac{\xi}{|\xi|}\times \frac{\eta}{|\eta|}\big)^\alpha \phi_{  l }(\frac{\xi}{|\xi|}-\mu\frac{\eta}{|\eta|}) \widehat{f_{k_1} }(\xi-\eta) \widehat{g_{k_2}}(\eta) d \eta,\\
\end{split}
\ee

\begin{lemma}\label{bilinearest}

 For the    bilinear form defined in \eqref{2024feb19eqn1}, the following estimate holds, 
 \be\label{bilinearestsummary}
\begin{split}
\| T^l_{k;k_1, k_2}(f, g)\|_{L^2} \lesssim &  2^{\alpha l }\| m\|_{\mathcal{S}^\infty_{k,k_1,k_2}} \min\Big\{    \| \hat{f}(\xi) \psi_{[k_1-2,k_1+2]}(\xi) \|_{L^2_\xi}  \| \mathcal{F}^{-1}[e^{-it \nu |\xi|} \widehat{g}(\xi)\psi_{k_2}(\xi)]\|_{L^\infty_x}, \\
& \qquad  \|\hat{g}(\xi) \psi_{[k_2-2,k_2+2]}(\xi)\|_{L^2}  \| \mathcal{F}^{-1}[e^{-it \mu |\xi|} \widehat{f}(\xi)\psi_{k_1}(\xi)]\|_{L^\infty_x}, \\ 
& \qquad 2^{2k_2 +l} \| f\|_{L^2} \|\widehat{g}\|_{L^\infty_\xi},  \min\big\{2^{k+(k_1+k_2)/2+l}, 2^{k_2+k_1+l}\big\}  \| \widehat{f}\|_{L^\infty_{\xi}} \| g\|_{L^2}\Big\} .\\
\end{split}
\ee
\end{lemma}
\begin{proof}
For any $l\in (-\infty, -100)\cap \Z_{-}$,  we define $K_l:=\lfloor \pi 2^{5-l} \rfloor$. Moreover,  $  \forall i\in\{1, \cdots, K_l\}, $ denote $ \phi^l_i:=i 2^{l-4}$, and $\theta_i^l:=(\cos\phi_i^l, \sin \phi_i^l)$. Let
\be\label{angularpartition}
\varphi_{l,i}(\xi):=\frac{\psi_{\leq l-4}(\frac{\xi}{|\xi|}-\theta_i^l)}{\sum_{i=1,\cdots, K}\psi_{\leq l-4}(\frac{\xi}{|\xi|}-\theta_i^l)}, \quad \Longrightarrow \sum_{i\in \{1,\cdots, K_l\}} \varphi_{l,i}(\xi) = 1. 
\ee
Therefore $\{\varphi_{l,i}(\xi)\}_{i=1}^{K^l}$ is a partition of unity with finite overlap. Moreover, $supp(\varphi_{l,i}(\cdot))$ lies in an angular sector centered at $\theta_i^l\in \mathbb{S}^1$. 

We now examine the kernel of symbols in \eqref{2024feb19eqn1} together with the angle-dependent cutoff functions defined in \eqref{angularpartition}. More precisely, we define  
\[
\begin{split}
 \mathcal{K}_{k;k_1,k_2}^{ \tilde{l},i_1;l,i_2}(x,y):=\int_{\R^2} \int_{\R^2} e^{ix\cdot\xi+ i y\cdot \eta }&   \big(\frac{\xi}{|\xi|}-\mu\frac{\eta}{|\eta|}\big)^\alpha   \phi_{l}(\frac{\xi}{|\xi|}-\mu\frac{\eta}{|\eta|})\\ 
 &\times m(\xi, \eta)\psi_{k}(\xi) \psi_{k_1}(\xi-\eta)  \psi_{k_2}(\eta)\varphi_{\tilde{l},i_1}(\xi)  \varphi_{l,i_2}(\eta) d \eta d \xi .    
\end{split}
\]
Note that, for any $x,y\in \R^2$, we have 
\[
  e^{i x\cdot\xi} = \frac{\theta_{i_1}^{\tilde{l}}\cdot \nabla_\xi ( e^{i x\cdot\xi}  )}{i(\theta_{i_1}^{\tilde{l}}\cdot x)}, \quad   e^{i x\cdot\xi} = \frac{\theta_{i_1}^{\tilde{l}}\times  \nabla_\xi ( e^{i x\cdot\xi}  )  }{i(\theta_{i_1}^{\tilde{l}}\times x) },\quad e^{i y\cdot\eta} = \frac{\theta_{i_2}^{l}\cdot \nabla_\eta ( e^{i  y\cdot\eta}  )}{i(\theta_{i_2}^{l}\cdot y)}, \quad   e^{i  y\cdot\eta} = \frac{\theta_{i_2}^{l}\times  \nabla_\eta ( e^{i  y\cdot\eta}  )  }{i(\theta_{i_2}^{l}\times y) }. 
\]
Moreover, we observe that  the directional derivatives $\theta_{i_1}^{\tilde{l}}\cdot \nabla_\xi $ and    $\theta_{i_2}^{l}\cdot \nabla_\eta$  of the symbol lose less. 

 To exploit the benefit of this fact, we first rewrite the oscillation phase using these identities and proceed with integration by parts along the $\{\theta_{i_1}^{\tilde{l}}, (\theta_{i_1}^{\tilde{l}})^{\bot}\}$ directions for $\xi$ and $\{\theta_{i_2}^{l}, (\theta_{i_2}^{l})^{\bot}\}$ directions for $\eta$ in twenty iterations. As a result,  for the kernel $ \mathcal{K}_{k;k_1,k_2}^{l;\tilde{l}}(x,y)$, we have

\be\label{2024feb19eqn6}
\begin{split}
\big|   \mathcal{K}_{k;k_1,k_2}^{ \tilde{l},i_1;l,i_2}(x,y)\big|
 &\lesssim 2^{\alpha l  }\| m(\xi, \eta)\|_{\mathcal{S}^\infty_{k,k_1,k_2}} 2^{2\min\{k,k_1\}+\min\{l,\tilde{l}\}}(1+2^{\min\{k,k_1\}}|\theta_{i_1}^{\tilde{l}}\cdot x|)^{-4}\\
 &\quad \times (1+2^{k+\min\{l,\tilde{l}\}}|\theta_{i_1}^{l_1}\times x|)^{-4} 2^{2k_2+l} (1+2^{k_2}|\theta_{i_2}^{l}\cdot x|)^{-4}(1+2^{k_2+l}|\theta_{i_2}^{l}\times y|)^{-4}.\\
 \end{split}
\ee

Applying the above partition of unity, we decompose the bilinear form $T^l_{k;k_1, k_1}(f, g) $ as follows, 
\be\label{angulardecomp}
\begin{split}
T^l_{k;k_1, k_2}(f, g)&=\sum_{\begin{subarray}{c}
i_1  \in \{1,\cdots, K^{(k_2-k_1)+l}\},i_2  \in \{1,\cdots, K^{l}\}\\
|\theta_{i_1}^{(k_2-k_1)+l}-\theta_{i_2}^{ l}|\leq  2^{l+10}\\
\end{subarray}} T^{l;i_0,i_1, i_2}_{k;k_1, k_2}(f, g), \\
 T^{l;i_0,i_1, i_2}_{k;k_1, k_2}(f, g)&=  \int_{\R^2}  e^{it (|\xi|-\mu|\xi-\eta|-\nu |\eta|)} \psi_{k}(\xi) \psi_{k_1}(\xi-\eta)\psi_{k_2}(\eta)  \big(\frac{\xi}{|\xi|}-\mu\frac{\eta}{|\eta|}\big)^\alpha \phi_{l}(\frac{\xi}{|\xi|}-\mu\frac{\eta}{|\eta|}) \\ 
&\quad  \times \varphi_{k_2-k_1+l,i_1}(\xi)  \varphi_{l,i_2}(\eta) \varphi_{k_2-k_1+l, i_1}(\pm(\xi-\eta)) m(\xi, \eta)\widehat{f^{\mu}}(\xi-\eta) \widehat{g^{\nu}}(\eta) d \eta.\\
\end{split}
\ee
In the above decomposition, we used the fact that $|\frac{\xi}{|\xi|}\times \frac{\xi-\eta}{|\xi-\eta|}|\sim 2^{k_2-k_1+l}. $ Due to   the orthogonality in $L^2$, from the   estimate of kernel in \eqref{2024feb19eqn6},  we have
\be\label{2024feb19eqn11}
\begin{split}
\|T^l_{k;k_1, k_2}(f, g)\|_{L^2_\xi}^2& \lesssim  \sum_{\begin{subarray}{c}
i_1  \in \{1,\cdots, K^{(k_2-k_1)+l}\},i_2  \in \{1,\cdots, K^{l}\}\\
|\theta_{i_1}^{(k_2-k_1)+l}-\theta_{i_2}^{ l}|\leq  2^{l+10}\\
\end{subarray}}  \|T^{l;i_0,i_1, i_2}_{k;k_1, k_2}(f, g)\|_{L^2_\xi}^2 \\
 &\lesssim  \sum_{\begin{subarray}{c}
i_1  \in \{1,\cdots, K^{(k_2-k_1)+l}\},i_2  \in \{1,\cdots, K^{l}\}\\
|\theta_{i_1}^{(k_2-k_1)+l}-\theta_{i_2}^{ l}|\leq  2^{l+10}\\
\end{subarray}}  2^{2\alpha l }  \min\big\{  \| \hat{g}(\xi)\varphi_{l,i_2}(\xi)\psi_{[k_2-2,k_2+2]}(\xi) \|_{L^2_\xi}^2\\ 
&\quad \times   \|\mathcal{F}_{\xi\rightarrow x}^{-1}[e^{-it \mu |\xi|
}\varphi_{k_2-k_1+l, i_1}(\pm \xi)\hat{f}(\xi)\psi_{k_1}(\xi)  ]\|_{L^\infty_x}^2   ,  \|\varphi_{k_2-k_1+l, i_1}(\pm \xi)\hat{f}(\xi)\psi_{k_1}(\xi) \|_{L^2_\xi}^2 \\
&  \quad \times  \| \mathcal{F}_{\xi\rightarrow x}^{-1}[e^{-it \nu |\xi|
} \hat{g}(\xi)\varphi_{l,i_2}(\xi)\psi_{[k_2-2,k_2+2]}(\xi) \|_{L^\infty_x}^2 \big\} \| m(\xi, \eta)\|_{\mathcal{S}^\infty_{k,k_1,k_2}}^2  \\ 
& \lesssim 2^{2\alpha l } \| m(\xi, \eta)\|_{\mathcal{S}^\infty_{k,k_1,k_2}}^2   \big(\min\big\{ 2^{k_2+k_1+l}   \| \hat{f}(\xi)\psi_{k_1}(\xi)  \|_{L^\infty_\xi} \| \hat{g}(\xi) \psi_{[k_2-2,k_2+2]}(\xi) \|_{L^2_\xi},  \\
&  \quad  2^{2k_2+l}   \| \hat{g}(\xi) \psi_{[k_2-2,k_2+2]}(\xi) \|_{L^\infty_\xi}  \| \hat{f}(\xi)\psi_{k_1}(\xi)  \|_{L^2_\xi}  \big\}\big)^2. 
\end{split}
\ee

Note that if $|k_1-k_2|\geq 10$, we have $|k-\max\{k_1,k_2\}|\leq 10$. To prove the final estimate of \eqref{bilinearestsummary}, it would be sufficient to consider the case $|k_1-k_2|\leq 10$. As in \eqref{2024feb19eqn11}, if instead of using the $L^2-L^\infty$-type bilinear estimate, we first use the $L^2\rightarrow L^1$-type Sobolev embedding and then use the $L^2-L^2$-type bilinear estimate. As a result, we have
\be
\|T^l_{k;k_1, k_2}(f, g)\|_{L^2_\xi}\lesssim  2^{\alpha l } \| m(\xi, \eta)\|_{\mathcal{S}^\infty_{k,k_1,k_2}} 2^{k+l/2}2^{(k_1+k_2+l)/2} \| \hat{f}(\xi) \psi_{k_1}(\xi) \|_{L^\infty_\xi}  \| \hat{g}(\xi) \psi_{[k_2-2,k_2+2]}(\xi) \|_{L^2_\xi}. 
\ee
Moreover, from the obtained estimate \eqref{2024feb19eqn11},   we have 
\be
\|T^l_{k;k_1, k_2}(f, g)\|_{L^2_\xi}\lesssim 2^{\alpha l } \| m(\xi, \eta)\|_{\mathcal{S}^\infty_{k,k_1,k_2}}   \| \hat{f}(\xi) \psi_{[k_1-2,k_1+2]}(\xi) \|_{L^2_\xi}  \| \mathcal{F}^{-1}[e^{-it \nu |\xi|} \widehat{g}(\xi)\psi_{k_2}(\xi)]\|_{L^\infty_x}. 
\ee
Alternatively, if we  apply  the  partition of unity with $i_1,i_2  \in \{1,\cdots, K^{ l}\}, $  for $\xi,\eta$, we have 
\be\label{2024march27eqn21}
\begin{split}
\|T^l_{k;k_1, k_2}(f, g)\|_{L^2_\xi}^2 & \lesssim  \sum_{\begin{subarray}{c}
i_1,i_2  \in \{1,\cdots, K^{l}\}, 
|i_1-i_2|\leq  2^{10}\\
\end{subarray}}  2^{2\alpha l }     \| \hat{g}(\xi)\varphi_{l,i_2}(\xi)\psi_{[k_2-2,k_2+2]}(\xi) \|_{L^2_\xi}^2 \\
&\quad \times  \|\mathcal{F}_{\xi\rightarrow x}^{-1}[e^{-it \mu |\xi|
} \hat{f}(\xi)\psi_{k_1}(\xi)  ]\|_{L^\infty_x}^2 \| m(\xi, \eta)\|_{\mathcal{S}^\infty_{k,k_1,k_2}}^2 \\
& \lesssim 2^{2\alpha l } \| m(\xi, \eta)\|_{\mathcal{S}^\infty_{k,k_1,k_2}}^2    \| \hat{g}(\xi) \psi_{[k_2-2,k_2+2]}(\xi) \|_{L^2_\xi}^2  \|\mathcal{F}_{\xi\rightarrow x}^{-1}[e^{-it \mu |\xi|
} \hat{f}(\xi)\psi_{k_1}(\xi)  ]\|_{L^\infty_x}^2 .
\end{split}
\ee
To sum up, our desired estimates in \eqref{bilinearestsummary} hold from the above obtained   estimates \eqref{2024feb19eqn11}--\eqref{2024march27eqn21}. 
\end{proof}

\subsection{The estimate of $Z_u(t)$ }
\begin{lemma}\label{decayvelinf}
Let $t\in [2^{m-1},2^m]\subset [0, T], m\in \Z_{+}$, s.t., $m\gg 1$. Under the bootstrap assumption \eqref{bootstrap}, the following improved energy estimate for the velocity field $u$ holds, 
\be\label{2024march1eqn1}
Z_{u}(t)\lesssim 2^{-m}\epsilon_0. 
\ee
\end{lemma}
\begin{proof}
 Recall the normal form we did in \eqref{norform} and the equation satisfied by $v$ in \eqref{normalform}. We proceed in two main steps as follows. 

 \medskip

\noindent \textbf{Step 1.}\quad  The estimate of  $v$-part.

\medskip

  Note that, since $Q(u, u)$ has one derivative outside, which contributes one degree of smallness, from the $L^2-L^\infty$-type bilinear estimate, we have 
 \[
 \forall k \in Z, \quad \|P_k Q(u(s), u(s))\|_{L^2}\lesssim 2^{k-10k_{+}} \min\{2^k\langle s\rangle^{-1+2\delta}, \langle s \rangle^{-3/2+\delta}\}\epsilon_1^2. 
 \]  
From the above estimate and the estimate of heat kernel, we have 
\be\label{2024march1eqn2}
\begin{split}
&\sum_{k\in \Z} 2^{k+4k_{+}}\|\int_0^t e^{(t-s)\Delta} P_kQ(u(s), u(s)) d s \big\|_{L^2_x}\\
&\lesssim \sum_{k\in \Z}  \int_0^t 2^{2k-5k_{+}}(1+2^k|t-s|)^{-10}  \min\{2^k\langle s\rangle^{-1+2\delta}, \langle s \rangle^{-3/2+\delta}\}\epsilon_1^2 d s\lesssim \langle t\rangle^{-3/2+3\delta}\epsilon_0. 
\end{split}
\ee

Now we estimate the contribution of $A_{\mu, \nu}(\big(\mathcal{N}_\phi \big)^{\mu}, U^{\nu})$. Note that, the rough estimate  obtained for $\mathcal{N}_\phi$ in \eqref{localizedphi} is almost sufficient. We first rule out most of cases and then focus on the main scenario.

By using the $L^2_x-L^\infty_x$-type bilinear estimate, and   the $L^\infty_\xi$-estimate of $(u, U)$,  we have
\[
\| P_{k} A_{\mu, \nu}(\big(P_{k_1}\mathcal{N}_\phi \big)^{\mu}, P_{k_2}U^{\nu})\|_{L^2_x}\lesssim 2^{-(N_0-5)k_{+}}\langle  s\rangle^{-3/2+3\delta}\epsilon_1^2. 
\]
The above estimate allows us to rule out the case $k\geq 0$ as follows, 
\be\label{2024march1eqn4}
\begin{split}
&\sum_{k\in \Z_{+}} 2^{k+4k_{+}}\|\int_0^t e^{(t-s)\Delta} P_k A_{\mu, \nu}(\big(P_{k_1}\mathcal{N}_\phi \big)^{\mu}, P_{k_2}U^{\nu})  d s \big\|_{L^2_x}\\
&\lesssim \sum_{k\in \Z_+}  \int_0^t 2^{-(N_0-10)k_{+}}(1+2^k|t-s|)^{-10}   \langle  s\rangle^{-3/2+3\delta}\epsilon_1^2 d s\lesssim \langle t\rangle^{-3/2+3\delta}\epsilon_0. 
\end{split}
\ee

Now, we focus on the case $k\leq 0.$ From the $L^2_x-L^2_x$-type bilinear estimate,   the obtained estimate \eqref{localizedphi} and the $L^\infty_\xi$-estimate of $U$,   we first rule out the High $\times$ Low and Low $\times$ High type interaction as follows, 
\be\label{2024march1eqn51}
\begin{split}
&\sum_{(k_1,k_2)\in \chi_k^2\cup \chi_k^3} \| P_{k} A_{\mu, \nu}(\big(P_{k_1}\mathcal{N}_\phi \big)^{\mu}, P_{k_2}U^{\nu})\|_{L^1_x}\lesssim 2^{5k/4 }\langle  s\rangle^{-1}, \\
\Longrightarrow  & \sum_{k\in \Z_{-}}\sum_{(k_1,k_2)\in \chi_k^2\cup \chi_k^3} 2^{k+4k_{+}} \| \int_{0}^t e^{(t-s)\Delta}    P_{k} A_{\mu, \nu}(\big(P_{k_1}\mathcal{N}_\phi(s) \big)^{\mu}, P_{k_2}U^{\nu}(s)) d s\|_{L^2_x}\\ 
&\lesssim  \sum_{k\in \Z_{-}} \int_0^t 2^{9k/4 } (1+2^{2k}|t-s|)^{-10} \langle  s\rangle^{-1}\epsilon_1^2 d s\lesssim \langle t \rangle^{-1}\epsilon_0. \\
\end{split}
\ee 

Lastly, we consider  the High $\times$ High type interaction. For this case, we use different strategies for different pieces of $\mathcal{N}_\phi$. Recall \eqref{mainsysnew}.  We decompose  $\mathcal{N}_\phi$ into two pieces as follows, 
\be\label{nolinearityphi}
\begin{split}
\mathcal{N}_\phi&= \mathcal{N}_\phi^1+ \mathcal{N}_\phi^2, \quad   \mathcal{N}_\phi^1= \sum_{k\in \Z} -2 P_{\leq k-10}(u_i) \p_i P_{k}(\p_t \phi), \quad 
\mathcal{N}_\phi^2=  \mathcal{N}_\phi - \mathcal{N}_\phi^1. 
\end{split}
\ee
For $\mathcal{N}_\phi^2$ part, the following improved estimate holds since the $L^2_x$ of $\nabla u$ decay faster than $u$,
\be
\| P_k\mathcal{N}_\phi^2\|_{L^2}\lesssim 2^{k-3k_{+}} \langle s \rangle^{-3/2+2\delta}\epsilon.
\ee
From the above estimate and  the estimate of symbol in \eqref{symbolnormest}, we have 
\be\label{2024march1eqn52}
\begin{split}
&\sum_{k\in \Z_{-}}\sum_{(k_1,k_2)\in \chi_k^1} 2^{k+4k_{+}} \| \int_{0}^t e^{(t-s)\Delta}    P_{k} A_{\mu, \nu}(\big(P_{k_1}\mathcal{N}_\phi^2(s) \big)^{\mu}, P_{k_2}U^{\nu}(s))  ds\|_{L^2_x}\\ 
&\lesssim  \sum_{k\in \Z_{-}} \int_0^t 2^{2k } (1+2^{2k}|t-s|)^{-10} \langle  s\rangle^{-3/2+2\delta}\epsilon_1^2 d s\lesssim \langle t \rangle^{-1}\epsilon_0. \
\end{split}
\ee

It remains to consider the contribution of $\mathcal{N}_\phi^1.$ For this case, as in \eqref{2024feb19eqn41}, we use super localized decay estimate for the half wave. More precisely, we first consider the case when the frequency of $u$ in $\mathcal{N}_\phi^1$ is less than $2^k$. For this case,   we decompose the high frequency inputs into fine pieces as follows, 
\be\label{2024march1eqn21}
\begin{split}
&\|P_kA_{\mu, \nu}(P_{k_1}\big( P_{\leq k}(u_i)  P_{k_1'}(\p_t  \p_i \phi) \big)^{\mu}, P_{k_2}U^{\nu}(s))\|_{L^1_x}\\
&\lesssim \sum_{\begin{subarray}{c}
 n_1, n_2\in [2^{k_1-{k}-5}, 2^{k_1-{k}+5}]\cap \Z\\ 
 |n_1-n_2|\leq 2^{10}\\
 \end{subarray} } \|P_kA_{\mu, \nu}(P_{k_1}\big( P_{k'}(u_i)   P_{k_1'; {k},n_1}(\p_t \p_i \phi) \big)^{\mu}, P_{k_2; {k}, n_2}U^{\nu}(s))\|_{L^1_x},
\end{split}
\ee
where $|k_1'-k_1|\leq 5$ and    $P_{k_2;{k},n_1}$ denotes the Fourier multiplier operator with the Fourier symbol $\psi_{k_2}(\xi) \psi_{{k}}(|\xi|-2^{{k}}m)$. 

From the $L^2-L^2-L^\infty_x$ type multilinear estimate,  the  estimate \eqref{2024feb16eqn31},     the estimate of symbol in   \eqref{symbolnormest}, and the super localized decay estimate \eqref{decayestimate} in Lemma \ref{decaylemma},   we have 
\be\label{2024march1eqn22}
\begin{split}
\eqref{2024march1eqn21} & \lesssim  \sum_{\begin{subarray}{c}
 n_1, n_2\in [2^{k_1-{k}-5}, 2^{k_1-{k}+5}]\cap \Z\\ 
 |n_1-n_2|\leq 2^{10}\\
 \end{subarray} } \langle s\rangle^{-1/2} 2^{ -k_{1,-}} \|   P_{k_1';{k},n_1}(\p_t \p_i \phi)  \|_{L^2_x} \|P_{\leq k}(u)\|_{L^2}\\ 
 &\quad \times \big[ 2^{ {k} }\epsilon_1 + 2^{-m/4+3k_1/4+{k}/2} \|\nabla_\xi \widehat{V}(\xi) \psi_{k_2}(\xi)\psi_{{k}}(|\xi| - 2^{{k}} n_2 ) \|_{L^2} \big]\\
 &\lesssim  2^{  {k}/2+k_1-4k_{1,+}} \epsilon_1^2 \langle s \rangle^{-1+\delta}. 
\end{split}
\ee
Similarly, for any $k'\in [k,k_1-10]$,  we have
\[
\|P_kA_{\mu, \nu}(P_{k_1}\big( P_{k'}(u_i)  P_{k_1'}(\p_t  \p_i \phi) \big)^{\mu}, P_{k_2}U^{\nu}(s))\|_{L^1_x}\lesssim \langle s \rangle^{-1/2} 2^{k'/2+k_1-4k_{1,+}}\| P_{k'}(u)\|_{L^2}.
\]
Moreover, from $L^2-L^\infty_x-L^\infty_x$ type multilinear estimate, we have 
\[
\sum_{k'\leq k_1-10}\|P_kA_{\mu, \nu}(P_{k_1}\big( P_{k'}(u_i)  P_{k_1'}(\p_t  \p_i \phi) \big)^{\mu}, P_{k_2}U^{\nu}(s))\|_{L^1_x}\lesssim \langle s \rangle^{-3/2} 2^{k_1-3k_{1,+}}\epsilon_1^2. 
\]
From the above obtained  estimates, we have 
\be\label{2024march17eqn2}
\begin{split}
&\sum_{k\in \Z_{-}}\sum_{(k_1,k_2)\in \chi_k^1} 2^{k+4k_{+}} \| \int_{0}^t e^{(t-s)\Delta}    P_{k} A_{\mu, \nu}(\big(P_{k_1}\mathcal{N}_\phi^1(s) \big)^{\mu}, P_{k_2}U^{\nu}(s))  ds \|_{L^2_x}\\
&\lesssim \sum_{k\in \Z_{-}} \int_0^t \epsilon_1^2 2^k(1+2^{2k}|t-s|)^{-10} \big[\min\big\{ 2^{3k/2} \langle s \rangle^{-1+\delta}, \langle s \rangle^{-3/2} \big\} + 2^k \langle s\rangle^{-5/4+2\delta} \big] d s \lesssim  \langle t \rangle^{-1}\epsilon_1^2. 
\end{split}
\ee
To sum up, after combining the obtained estimates \eqref{2024march1eqn4},  \eqref{2024march1eqn51}, and \eqref{2024march17eqn2}, we have 
\be
\sum_{k\in \Z} 2^{k+4k_{+}} \| \int_{0}^t e^{(t-s)\Delta}    P_{k} A_{\mu, \nu}(\big(\mathcal{N}_\phi^1(s) \big)^{\mu},  U^{\nu}(s)) d s\|_{L^2_x} \lesssim \langle t\rangle^{-1}\epsilon_1^2. 
\ee
Due to symmetry of the bilinear form $A_{\mu, \nu}(\cdot, \cdot)$, from the above estimate and the obtained estimate \eqref{2024march1eqn2}, we have
\be\label{2024march3eqn21}
\sum_{k\in \Z} 2^{k+4k_{+}} \| \int_{0}^t e^{(t-s)\Delta}    P_{k}\mathcal{N}_v(s)   ds \|_{L^2_x} \lesssim \langle t\rangle^{-1}\epsilon_1^2. 
\ee

 \medskip

\noindent \textbf{Step 2.}\quad    The estimate of  the normal form part.

 \medskip

Recall \eqref{2024feb18eqn1} and \eqref{nullstrucdec}. We first use trivial estimates to rule out the very low frequency case and the relatively high frequency case. More precisely, from the estimate of symbol in \eqref{symbolnormest},  the $L^2_x\rightarrow L^1_x$ type Sobolev embedding and the $L^2_x-L^2_x$-type bilinear estimate, we have
\be
\begin{split}
&\sum_{k\in \Z, k\notin[-m/2-10\delta m, 2m/N_0]} 2^{k+4k_+}\| P_k A_{\mu, \nu}(U^{\mu}(t), U^{\nu}(t))\|_{L^2}\\ 
&\lesssim \sum_{k\in \Z, k\notin[-m/2-10\delta m, 2m/N_0]} 2^{2k-(N_0-10)k_++\delta m } \epsilon_1^2\lesssim 2^{-m}\epsilon_0.\\
\end{split} 
\ee

Now, we focus on the case $k\in[-m/2-10\delta m, 2m/N_0]\cap \Z$. To exploit the wave-wave-wave-type null structure, we localize the angle between $\xi $ and $\pm \eta$ as follows, 
\be\label{2024march3eqn19}
\begin{split}
&P_k A_{\mu, \nu}(U^{\mu}(t), U^{\nu}(t))=  \sum_{(k_1,k_2)\in \chi_k^1\cup \chi_k^2\cup\chi_k^3}\sum_{l\in [\bar{l},2]} T_{k;k_1,k_2}^{l}(t,x),\quad \bar{l}:=-m/2-\beta(k,k_1,k_2)/2, \\
& T_{k;k_1,k_2}^{l}(t,x):= \int_{\R^3}\int_{\R^3} e^{ix\cdot\xi + i t\Phi^{\mu, \nu}(\xi, \eta)}  a^l_{\mu, \nu}(\xi-\eta, \eta)  \widehat{V^\mu}(t, \xi-\eta)\widehat{V^\nu}(t,  \eta) d\eta d \xi,\\
&a^l_{\mu, \nu}(\xi-\eta, \eta):= a_{\mu, \nu}(\xi-\eta, \eta) \varphi_{l;\bar{l}}(\frac{\xi}{|\xi|}\times  \frac{\eta}{|\eta|}), \\
& \beta(k,k_1,k_2):= { k \mathbf{1}_{(k_1,k_2)\in \chi_k^1 } + (2k-k_1)\mathbf{1}_{(k_1,k_2)\in \chi_k^2} + k_2 \mathbf{1}_{(k_1,k_2)\in \chi_k^3} }. 
\end{split}
\ee

Based on the possible range of $(k_1, k_2)$, we proceed in sub-steps as follows. 

 \medskip

  \textbf{Step 2A.}\quad  If $(k_1,k_2)\in \chi_k^1$, i.e., $|k_1-k_2|\leq 10.$

 \medskip

From the bilinear estimate \eqref{bilinearestsummary} in Lemma \ref{bilinearest}, and the estimate of symbol in \eqref{symbolnormest}, we have 
\be\label{2024march3eqn11}
\|T_{k;k_1,k_2}^{\bar{l}}(t,x)\|_{L^2_x}\lesssim 2^{-k_{1,-}+2\bar{l}} 2^{k+4k_1/3-10k_{1,+}}\epsilon_1^2\lesssim 2^{-m+k_1/3-8k_{1,+}}\epsilon_1^2.
\ee
For $T_{k;k_1,k_2}^{l}(t,x)$, we do  integration by parts in $\eta$ once. As a result, we have
 \[
\begin{split}
T_{k;k_1,k_2}^{l}(t,x)&:=T_{k;k_1,k_2}^{l;1}(t,x)+ T_{k;k_1,k_2}^{l;2}(t,x)\\
T_{k;k_1,k_2}^{l;1}(t,x)& :=\sum_{\tilde{\mu}\in \{+,-\}}\int_{\R^3}\int_{\R^3} e^{ix\cdot\xi + i t\Phi^{\mu, \nu}(\xi, \eta)} \nabla_\eta\cdot \big[ \frac{i\nabla_\eta \Phi^{\mu, \nu}(\xi, \eta)}{t |\nabla_\eta \Phi^{\mu, \nu}(\xi, \eta)|^2} a^l_{\mu, \nu}(\xi-\eta, \eta) \big]   \widehat{V^\mu}(t, \xi-\eta)\widehat{V^\nu}(t,  \eta)  d\eta d \xi, \\
T_{k;k_1,k_2}^{l;2}(t,x)& := \int_{\R^3}\int_{\R^3} e^{ix\cdot\xi + i t\Phi^{\mu, \nu}(\xi, \eta)} \nabla_\eta\cdot \big[ \widehat{V^\mu}(t, \xi-\eta)\widehat{V^\nu}(t,  \eta)\big]   \frac{i\nabla_\eta \Phi^{\mu, \nu}(\xi, \eta)}{t |\nabla_\eta \Phi^{\mu, \nu}(\xi, \eta)|^2}   a^l_{\mu, \nu}(\xi-\eta, \eta)    d\eta d \xi. 
\end{split}
\]
For $T_{k;k_1,k_2}^{l;1}(t,x)$, we do integration by parts in $\eta$ again. As a result, we have 
\[
\begin{split}
T_{k;k_1,k_2}^{l;1}(t,x)& :=\int_{\R^3}\int_{\R^3} e^{ix\cdot\xi + i t\Phi^{\mu, \nu}(\xi, \eta)}  \nabla_\eta\cdot \big[ \frac{i\nabla_\eta \Phi^{\mu, \nu}(\xi, \eta)}{t |\nabla_\eta \Phi^{\mu, \nu}(\xi, \eta)|^2}  \nabla_\eta\cdot \big[ \frac{i\nabla_\eta \Phi^{\mu, \nu}(\xi, \eta)}{t |\nabla_\eta \Phi^{\mu, \nu}(\xi, \eta)|^2} \\ 
&\quad \times  a^l_{\mu, \nu}(\xi-\eta, \eta)\big]  \widehat{V^\mu}(t, \xi-\eta)\widehat{V^\nu}(t,  \eta)   \big] d\eta d \xi .
\end{split}
\]
Note that $ \big({(\xi-\eta)}/{|\xi-\eta|}\big)\times \big(\eta/|\eta|\big)= (\xi\times \eta)/\big(|\xi-\eta||\eta| \big)$. 
From the estimate of symbol in \eqref{symbolnormest}, and  the bilinear estimate \eqref{bilinearestsummary} in Lemma \ref{bilinearest},      we have 
\be\label{2024march3eqn12}
\begin{split}
\sum_{l\in (\bar{l}, 2]\cap \Z}\|T_{k;k_1,k_2}^{l;1}(t,x)\|_{L^2_x}
&\lesssim 2^{-\max\{k_{1}, k_2\}_{-}-2m-2(k+2l)}  2^{k+4k_1/3+2l-10k_{1,+}}\epsilon_1^2 \lesssim 2^{-m+k_1/3-8k_{1,+}}\epsilon_1^2.  
\end{split}
\ee
Moreover, from  the estimate of symbol in \eqref{symbolnormest},   the bilinear estimate \eqref{bilinearestsummary} in Lemma \ref{bilinearest},   the estimate \eqref{2024feb16eqn31}, and the decay estimate \eqref{decayestimate} in Lemma \ref{decaylemma},  we have 
\be\label{2024march3eqn13}
\begin{split}
\sum_{l\in (\bar{l}, 2]\cap \Z}\|T_{k;k_1,k_2}^{l;2}(t,x)\|_{L^2_x}
&\lesssim 2^{-\max\{k_{1}, k_2\}_{-}-m-(k-k_1+l)+l} 2^{-(1-\alpha) k_1+2\delta m} 2^{-m/2+k_1/2-8k_{1,+}} \\ 
& \lesssim 2^{-3m/2+3\delta m-k +(3/2+\alpha)k_1-8k_{1,+}}\epsilon_1^2. \\
\end{split}
\ee

 \medskip

  \textbf{Step 2B.}\quad   If $(k_1,k_2)\in \chi_k^2$, i.e., $k_1\leq k-10.$

  \medskip
By using the same strategy, from  the estimate of symbol in \eqref{symbolnormest},   the bilinear estimate \eqref{bilinearestsummary} in Lemma \ref{bilinearest}, and the decay estimate \eqref{decayestimate}, we have 
\be\label{2024march3eqn14}
\begin{split}
\|T_{k;k_1,k_2}^{\bar{l}}(t,x)\|_{L^2_x}&\lesssim 2^{-k_{+}+ 2l } 2^{k_1/3+2k-10k_{+}}\epsilon_1^2\lesssim 2^{-m+k_1/3-8k_{+}}\epsilon_1^2, \\
\sum_{l\in (\bar{l}, 2]\cap \Z}\|T_{k;k_1,k_2}^{l;1}(t,x)\|_{L^2_x}
&\lesssim \sum_{l\in (\bar{l}, 2]\cap \Z} 2^{-k_{+}-2m-2(2k-k_1+2l)+2l} 2^{k_1/3+2k-10k_{+}}\lesssim 2^{-m+k_1/3-8k_{+}}\epsilon_1^2, \\
\sum_{l\in (\bar{l}, 2]\cap \Z}\|T_{k;k_1,k_2}^{l;2}(t,x)\|_{L^2_x}
&\lesssim 2^{-k_{+}-m-(k-k_1+l)+l} 2^{-(1-\alpha)k_1+2\delta m} 2^{-m/2+k_2/2-8k_{2,+}}  \\ 
& \lesssim 2^{-5m/4 +\alpha k_1  -8k_{+}}\epsilon_1^2.  \\ 
\end{split}
\ee

 \medskip

  \textbf{Step 2C.}\quad  If $(k_1,k_2)\in \chi_k^2\cup\chi_k^3$,  i.e., $k_2\leq k-10.$

   \medskip

 By using the same strategy, from  the estimate of symbol in \eqref{symbolnormest},   the bilinear estimate \eqref{bilinearestsummary} in Lemma \ref{bilinearest}, and the decay estimate \eqref{decayestimate}, we have 
\be\label{2024march3eqn15}
\begin{split}
\|T_{k;k_1,k_2}^{\bar{l}}(t,x)\|_{L^2_x}&\lesssim 2^{-k_{+}+ 2l } 2^{4k_2/3+k-10k_{+}}\epsilon_1^2\lesssim 2^{-m+k_2/3-8k_{+}}\epsilon_1^2, 
\\
\sum_{l\in (\bar{l}, 2]\cap \Z}\|T_{k;k_1,k_2}^{l;1}(t,x)\|_{L^2_x}
&\lesssim \sum_{l\in (\bar{l}, 2]\cap \Z} 2^{-k_{+}-2m-2(k_2+2l)+2l} 2^{4k_2/3+k-10k_{+}}\lesssim 2^{-m+k_2/3-8k_{+}}\epsilon_1^2, \\
\sum_{l\in (\bar{l}, 2]\cap \Z}\|T_{k;k_1,k_2}^{l;2}(t,x)\|_{L^2_x}
&\lesssim  \sum_{l\in (\bar{l}, 2]\cap \Z} 2^{-k_{+}-m-l+l} 2^{-(1-\alpha)k_2+2\delta m-8k_{+}}\min\{ 2^{-m/2+k_1/2 }, 2^{k_2+l+2k_1/3}\}\\ 
& \lesssim 2^{-m/2-\alpha m/4+3\delta m +\alpha k_2/2 -8k_{+}}\epsilon_1^2.\\
\end{split}
\ee
 
Recall the decomposition \eqref{2024march3eqn19}. After combining the   obtained estimates (\ref{2024march3eqn11}--\ref{2024march3eqn15}), we have 
\be\label{2024march4eqn91}
\sum_{ k\in[-m/2-10\delta m, 2m/N_0]\cap \Z}2^{k+4k_+}\|P_k A_{\mu, \nu}(U^{\mu}(t), U^{\nu}(t))\|_{L^2}\lesssim 2^{-m}\epsilon_1^2. 
\ee
To sum up, our desired estimate \eqref{2024march1eqn1} holds from the above estimate and the obtained estimate \eqref{2024march3eqn21}. 
\end{proof}

\subsection{The estimate of $Z_\phi(t)$}
\begin{lemma}\label{Linfwavepart}
Let $t_1, t_2\in [2^{m-1},2^m]\subset [0, T], m\in \Z_{+}$, s.t., $m\gg 1$. Under the bootstrap assumption \eqref{bootstrap}, we have 
\be\label{2024march1eqn41}
\sup_{k\in \Z} 2^{k/3+ 10k_{+}} \|\int_{t_1}^{t_2} e^{i s  |\xi|} \widehat{\mathcal{N}_\phi}(s, \xi) \psi_k(\xi) d  s \|_{L^\infty_\xi} \lesssim 2^{-0.04 m +10\delta m}\epsilon_0. 
\ee
\end{lemma}
\begin{proof}

Recall \eqref{highphi} and \eqref{mainsysnew}. We decompose $\mathcal{N}_\phi$ into two parts as follows, 
\be\label{2024march17eqn5}
\begin{split}
\mathcal{N}_\phi&= M_\phi + G_\phi, \quad M_\phi:= -2u\cdot \nabla_x\p_t \phi -\Delta u\cdot \nabla_x \phi, \quad G_\phi:=-(\p_t -\Delta)u\cdot \nabla_x \phi -u\cdot \nabla(u\cdot \nabla_x\phi). 
\end{split}
\ee

More precisely, from the $L^2_x-L^2_x-L^\infty_x$ type multilinear estimate, we have 
\be\label{2024march4eqn44}
\sum_{|\alpha|\leq 10}\|\nabla_x^\alpha\p_i(u_i(s) u_j(s) \p_j \phi(s)) \|_{L^1_\xi} \lesssim \langle s \rangle^{-3/2+2\delta}\epsilon_1^3. 
\ee
Moreover, from  the estimate \eqref{2024march4eqn81} in Lemma \ref{nullcubictermLin} and   the $L^2_x-L^2_x-L^\infty_x$ type multilinear estimate,  we have
\be\label{2024march3eqn31}
\begin{split}
  &  \sup_{k\in \Z}2^{k/3+ 10k_{+}}\|\mathcal{F}_{x\rightarrow \xi}\big[ Q(u, u)\cdot \nabla_x \phi\big](\xi)\psi_k(\xi) \|_{L^\infty_\xi}\lesssim \langle s \rangle^{-3/2+5\delta}\epsilon_1^3, \\ 
    &  \sup_{k\in \Z}2^{k/3+ 10k_{+}}\|\mathcal{F}_{x\rightarrow \xi}\big[ \widetilde{Q}(\phi, \phi)\cdot \nabla_x \phi\big](\xi)\psi_k(\xi) \|_{L^\infty_\xi}\lesssim \langle s \rangle^{-1.04+5\delta}\epsilon_1^3, \\ 
\Longrightarrow \quad &\sup_{k\in \Z}2^{k/3+  10k_{+}}\|\mathcal{F}_{x\rightarrow \xi}\big[(\p_t -\Delta)u_i\p_i \phi\big](\xi)\psi_k(\xi) \|_{L^\infty_\xi}\lesssim \langle s \rangle^{-5/4+5\delta}\epsilon_1^3. \\
\end{split}
\ee
From the above estimate and the obtained estimate \eqref{2024march4eqn44}, we have 
\be\label{2024march4eqn45}
\sup_{k\in \Z} 2^{k/3+ 10k_{+}}\|\mathcal{F}_{x\rightarrow \xi}\big[\int_{t_1}^{t_2} e^{is\d}G_\phi d s\big](\xi) \psi_k(\xi)\|_{L^\infty_\xi} \lesssim 2^{-0.04m+10\delta m}\epsilon_1^2. 
\ee

Now, we focus on the contribution of the main part $M_\phi$. Note that, from the Duhamel's formula, on the Fourier side, for any $\xi\in supp(\psi_k(\cdot))$,  we have  
\be\label{2024march1eqn49}
\begin{split}
&\mathcal{F}_{x\rightarrow \xi}\big[\int_{t_1}^{t_2} e^{is\d}M_\phi d s\big](\xi) \psi_k(\xi) =\sum_{\mu\in \{+,-\}}\sum_{(k_1,k_2)\in \chi_k^1\cup\chi_k^2\cup \chi_k^3}\int_{t_1}^{t_2} I^\mu_{k;k_1, k_2}(s, \xi) d s, \\
 &I^\mu_{k;k_1, k_2}(s, \xi):= \int_{\R^3} e^{is(|\xi|-\mu|\eta|) } m^{\mu;j}_{k,k_1,k_2}(\xi, \eta)\cdot \widehat{u_j}(s, \xi-\eta)\widehat{V^{\mu}}(s, \eta) d \eta, \\
\end{split}
\ee
where $m^{\mu;j}_{k,k_1,k_2}(\xi, \eta):= i\eta_j\big(1+ c_\mu |\xi-\eta|^2|\eta|^{-1}\big)\psi_k(\xi) \psi_{k_1}(\xi-\eta)\psi_{k_2}(\eta).$ By  using the energy norm and the volume of support, we can rule out the very low and relative high frequency as follows, 
\be\label{2024march4eqn47}
\begin{split}
&\sum_{(k_1, k_2)\in  \Z^2/[-2m/3 , 2m/N_0]^2}  2^{k/3+ 10 k_{+}} \|I^\mu_{k;k_1, k_2}(s, \xi)\|_{L^\infty_\xi}\lesssim \sum_{(k_1, k_2)\in  \Z^2/[-2m/3, 2m/N_0]^2}  2^{k/3+ 10 k_{+}-m/2+\delta m }\\
 &\qquad \times \min\{2^{2\min\{k_1, k_2\}/3}, 2^{-(N_0-5)\max\{k_1,k_2\}_{} }\}\epsilon_1^2\lesssim 2^{-1.1m }\epsilon_1^2. \\
 \end{split}
\ee

Now, we focus on the case $k_1,k_2\in [-2m/3 , 2m/N_0]$. For this case,   we do integration by parts in $\eta$ once for $I^\mu_{k;k_1, k_2}(t,\xi)$. As a result, we have 
\be\label{2024march4eqn21}
\begin{split}
&I^\mu_{k;k_1, k_2}(s, \xi):=  I^{\mu;1}_{k;k_1, k_2}(s, \xi)+  I^{\mu;2}_{k;k_1, k_2} (s, \xi), \\
&  I^{\mu;1}_{k;k_1, k_2}(s, \xi):= \int_{\R^3} e^{i s(|\xi|-\mu|\eta|) }\widehat{u_j}(s, \xi-\eta)  \nabla_\eta\cdot   \big[  \frac{i \mu \eta}{s|\eta|}  m^{\mu;j}_{k,k_1,k_2}(\xi, \eta)  \widehat{V^{\mu}}(s, \eta)\big] d \eta, \\
& I^{\mu;2}_{k;k_1, k_2}(s, \xi):= \int_{\R^3} e^{i s(|\xi|-\mu|\eta|) } \nabla_\eta\big( \widehat{u_j}(t, \xi-\eta)\big) \cdot   \big[  \frac{i \mu \eta}{s|\eta|}    m^{\mu;j}_{k,k_1,k_2}(\xi, \eta)  \widehat{V^{\mu}}(s, \eta)\big] d \eta.  \\
\end{split}
\ee
Note that,  
\[
\begin{split}
\nabla_\xi   \widehat{u}(s, \xi)
&= - \frac{\xi}{|\xi|^2}  \widehat{Su}(s, \xi) + \frac{(\xi_2,-\xi_1)}{|\xi|^2}(\xi_2\p_{\xi_1}-\xi_1\p_{\xi_2})   \widehat{ u}(s, \xi)  +  s \frac{\xi}{|\xi|^2}  \widehat{\p_s u}(s, \xi)\\
&= - \frac{\xi}{|\xi|^2}  \widehat{Su}(s, \xi) + \frac{(\xi_2,-\xi_1)}{|\xi|^2}    \widehat{\Omega u}(s, \xi)  +  s \frac{\xi}{|\xi|^2}  \widehat{\mathcal{N}_u}(s, \xi) - s \xi \widehat{u}(s, \xi).
\end{split}
\]
From the above equality, we decompose $I^{\mu;2}_{k;k_1, k_2}(s, \xi)$ into three parts as follows, 
\be\label{2024march4eqn22}
I^{\mu;2}_{k;k_1, k_2}(s, \xi):= I^{\mu;2;1}_{k;k_1, k_2}(s, \xi) +  I^{\mu;2;2}_{k;k_1, k_2}(s, \xi) + I^{\mu;2;3}_{k;k_1, k_2}(s, \xi) ,  
\ee
\[
\begin{split}
I^{\mu;2;1}_{k;k_1, k_2}(s, \xi) & :=  \int_{\R^3}   e^{is(|\xi|-\mu|\eta|) }      m^{\mu;j}_{k,k_1,k_2}(\xi, \eta)     \frac{i \mu \eta}{s|\eta|}  \cdot  \big[ - \frac{\xi-\eta}{|\xi-\eta|^2}  \widehat{Su_j}(s, \xi-\eta) + \frac{(\xi_2-\eta_2,-\xi_1+\eta_1)}{|\xi-\eta|^2}\\ 
&\times   \widehat{\Omega u_j}(s, \xi-\eta) + \frac{s(\xi-\eta)}{|\xi-\eta|^2}\mathcal{F}_{}[Q_j(u, u)](s,\xi- \eta) \big]  \widehat{V^{\mu}}(s, \eta) d \eta, \\
I^{\mu;2;2}_{k;k_1, k_2}(s, \xi) & :=  \int_{\R^3}  e^{is(|\xi|-\mu|\eta|) }   \big[  \frac{i \mu \eta\cdot(\xi-\eta)}{|\eta||\xi-\eta|^2}       m^{\mu;j}_{k,k_1,k_2}(\xi, \eta)  \widehat{V^{\mu}}(s, \eta)  \mathcal{F}_{}[\widetilde{Q}_j(\phi, \phi)](s,\xi- \eta)   d \eta, \\
I^{\mu;2;3}_{k;k_1, k_2}(s, \xi) & :=  \int_{\R^3}  e^{is (|\xi|-\mu|\eta|) }  \frac{i \mu \eta\cdot(\xi-\eta)}{|\eta| }      m^{\mu;j}_{k,k_1,k_2}(\xi, \eta)  \widehat{V^{\mu}}(s, \eta)  \widehat{u_j}(s,\xi- \eta)   d \eta, \\
\end{split}
\]
where we use $Q_j(\cdot, \cdot)$ and $\widetilde{Q}_j$ to denote the $j$-th component of the vector $Q(\cdot, \cdot)$ and $\widetilde{Q}(\cdot, \cdot).$

For $I^{\mu;2;3}_{k;k_1, k_2}(s, \xi)$, we repeat the procedure we did for $I_\mu(t, \xi)$. As a result, we have
\be\label{2024march4eqn24}
I^{\mu;2;3}_{k;k_1, k_2}(s, \xi)= \widetilde{I}^{\mu;1}_{k;k_1,k_2}(s, \xi) +\widetilde{I}^{\mu;2}_{k;k_1,k_2}( s, \xi) + \widetilde{I}^{\mu;3}_{k;k_1,k_2}(s, \xi)+\widetilde{I}^{\mu;4}_{k;k_1,k_2}(s, \xi), \\
\ee
\[
\begin{split}
\widetilde{I}^{\mu;1}_{k;k_1,k_2}(s,\xi) &=   \int_{\R^3}  e^{is (|\xi|-\mu|\eta|) }   \widehat{u_j}(s, \xi-\eta)  \nabla_\eta\cdot   \big[  \frac{ \eta (\xi-\eta)\cdot\eta}{s|\eta|^2}       m^j_\mu(\xi, \eta)  \widehat{V^{\mu}}(s, \eta)\big] d \eta,  \\
\widetilde{I}^{\mu;2}_{k;k_1,k_2}(s,\xi)  &= -  \int_{\R^3}  e^{is (|\xi|-\mu|\eta|) }         m^{\mu;j}_{k,k_1,k_2}(\xi, \eta) \frac{ \eta (\xi-\eta)\cdot\eta}{s|\eta|^2}           \cdot  \big[ \frac{\xi-\eta}{|\xi-\eta|^2}  \widehat{Su_j}(s, \xi-\eta) - \frac{(\xi_2-\eta_2,-\xi_1+\eta_1)}{|\xi-\eta|^2}\\ 
&\times ((\xi_2-\eta_2)\p_{\xi_1}-(\xi_1-\eta_1)\p_{\xi_2})   \widehat{\Omega u_j}(s, \xi-\eta) - \frac{s(\xi-\eta)}{|\xi-\eta|^2}\mathcal{F}_{}[Q_j(u, u)](s,\xi- \eta) \big]  \widehat{V^{\mu}}(s, \eta)  d \eta, \\
\widetilde{I}^{\mu;3}_{k;k_1,k_2}(s,\xi) &= -  \int_{\R^3}  e^{is (|\xi|-\mu|\eta|) }      \frac{ (\eta\cdot(\xi-\eta))^2}{|\eta|^2|\xi-\eta|^2}      m^{\mu;j}_{k,k_1,k_2}(\xi, \eta)  \widehat{V^{\mu}}(s, \eta)  \mathcal{F}_{}[\widetilde{Q}_j(\phi, \phi)](s,\xi- \eta)   d \eta, \\
\widetilde{I}^{\mu;4}_{k;k_1,k_2}(s,\xi) &= -\int_{\R^3} e^{is (|\xi|-\mu|\eta|) }     \frac{\big( \eta\cdot(\xi-\eta)\big)^2 }{|\eta|^2 }      m^{\mu;j}_{k,k_1,k_2}(\xi, \eta)   \widehat{V^{\mu}}(s, \eta)  \widehat{u_j}(s,\xi- \eta)   d \eta.
\end{split}
\]

For $\widetilde{I}^{\mu;4}_{k;k_1,k_2}(s,\xi)$, we do integration by parts in times $s$ once. As a result, we have
\be\label{2023march3eqn41}
\begin{split}
&\int_{t_1}^{t_2}  \widetilde{I}^{\mu;4}_{k;k_1,k_2}(s,\xi)  d s  =  End^{\mu}_{k;k_1,k_2}(t_1,t_2,\xi)  + \sum_{i=1,2,3}\mathcal{I}^{\mu;i}_{k;k_1,k_2}(t_1,t_2,\xi),\\
&End^{\mu}_{k;k_1,k_2}(t_1,t_2,\xi)  =  \sum_{j=1,2}(-1)^{j} \int_{\R^3} e^{it_j (|\xi|-\mu|\eta|) }  \frac{   m^j_\mu(\xi, \eta)  \widehat{V^{\mu}}(t_j,  \eta)  \widehat{u_j}(t_j,\xi- \eta)}{|\xi-\eta|^2 + i (|\xi|-\mu|\eta|)}   \frac{\big( \eta\cdot(\xi-\eta)\big)^2 }{|\eta|^2 }    d \eta,\\
&\mathcal{I}^{\mu;1}_{k;k_1,k_2}(t_1,t_2,\xi)  = \int_{t_1}^{t_2}  \int_{\R^3} e^{i s (|\xi|-\mu|\eta|) }  \frac{   m^j_\mu(\xi, \eta) \p_s \widehat{V^{\mu}}(s,  \eta)  \widehat{u_j}(s,\xi- \eta)}{|\xi-\eta|^2 + i (|\xi|-\mu|\eta|)}   \frac{\big( \eta\cdot(\xi-\eta)\big)^2 }{|\eta|^2 }    d \eta d  t,\\
&\mathcal{I}^{\mu;2}_{k;k_1,k_2}(t_1,t_2,\xi)   = \int_{t_1}^{t_2}  \int_{\R^3} e^{i s (|\xi|-\mu|\eta|) }  \frac{   m^j_\mu(\xi, \eta)  \widehat{V^{\mu}}(s,  \eta) \mathcal{F}[Q_j(u,u)](s,\xi- \eta)}{|\xi-\eta|^2 + i (|\xi|-\mu|\eta|)}   \frac{\big( \eta\cdot(\xi-\eta)\big)^2 }{|\eta|^2 }    d \eta d  t,\\
&\mathcal{I}^{\mu;3}_{k;k_1,k_2}(t_1,t_2,\xi)   = \int_{t_1}^{t_2}  \int_{\R^3} e^{i s (|\xi|-\mu|\eta|) }  \frac{   m^j_\mu(\xi, \eta)  \widehat{V^{\mu}}(s,  \eta) \mathcal{F}[\widetilde{Q}_j(\phi,\phi)](s,\xi- \eta)}{|\xi-\eta|^2 + i (|\xi|-\mu|\eta|)}   \frac{\big( \eta\cdot(\xi-\eta)\big)^2 }{|\eta|^2 }    d \eta d  t.\\
\end{split}
\ee

From the $L^2-L^2$ type bilinear estimate, the estimate \eqref{2024feb16eqn31}, we have
\be\label{2024march4eqn1}
\begin{split}
\sup_{k\in \Z} 2^{k/3+ 10k_{+}}\big( &\|  I^{\mu;1}_{k;k_1, k_2}(s, \xi) \psi_k(\xi)\|_{L^\infty_\xi} + \|  I^{\mu;2;1}_{k;k_1, k_2}(s, \xi) \psi_k(\xi)\|_{L^\infty_\xi} + \| \widetilde{I}^{\mu;1}_{k;k_1, k_2}(s, \xi) \psi_k(\xi)\|_{L^\infty_\xi} \\ 
&+ \| \widetilde{I}^{\mu;2}_{k;k_1, k_2}(s, \xi) \psi_k(\xi)\|_{L^\infty_\xi} \big)\lesssim \langle s \rangle^{-3/2+2\delta}\epsilon_1^2, \\
\sup_{k\in \Z} 2^{k/3+10k_{+}}\big(&\| End^{\mu}_{k;k_1,k_2}(t_1,t_2,\xi) \psi_k(\xi) \|_{L^\infty_\xi} + \|\mathcal{I}^{\mu;1}_{k;k_1,k_2}(t_1,t_2,\xi)  \psi_k(\xi)\|_{L^\infty_\xi} \\ 
&+ \| \mathcal{I}^{\mu;2}_{k;k_1,k_2}(t_1,t_2,\xi)  \psi_k(\xi)\|_{L^\infty_\xi}\big) \lesssim 2^{-0.04 m+10\delta}\epsilon_1^2. 
\end{split}
\ee
From the estimate \eqref{2024march4eqn81} in Lemma \ref{nullcubictermLin}, we have
\be\label{2024march4eqn3}
\begin{split}
\sup_{k\in \Z} 2^{k/3+10k_{+}}\big(&2^m \|  I^{\mu;2;2}_{k;k_1, k_2}(s, \xi) \psi_k(\xi)\|_{L^\infty_\xi} + 2^m \| \widetilde{I}^{\mu;3}_{k;k_1,k_2}(s,\xi) \psi_k(\xi)\|_{L^\infty_\xi} \\ 
& +  \| \mathcal{I}^{\mu;3}_{k;k_1,k_2}(t_1,t_2,\xi) \psi_k(\xi)\|_{L^\infty_\xi} \big) \lesssim 2^{-0.04 m +10\delta}\epsilon_1^2.
\end{split}
\ee
Recall the decomposition in \eqref{2024march4eqn21}, \eqref{2024march4eqn22},   \eqref{2024march4eqn24}, and \eqref{2023march3eqn41}.  To sum up, after combining the obtained estimates \eqref{2024march4eqn1} and \eqref{2024march4eqn3}, for any  $k_1,k_2\in [-m/2-10\delta m, 2m/N_0]\cap\Z$,  we have 
\be
 \sup_{k\in \Z} 2^{k/3+ 10 k_{+}} \|\int_{t_1}^{t_2} I^\mu_{k;k_1, k_2}(s, \xi)  ds \|_{L^\infty_\xi}\lesssim 2^{-m/2+10\delta m}\epsilon_1^2.
\ee
Recall decompositions in  \eqref{2024march1eqn51} and  \eqref{velocityeqn}. Hence our desired estimate \eqref{2024march1eqn41} holds from the above estimate, the obtained estimates \eqref{2024march4eqn45} and \eqref{2024march4eqn47}.
\end{proof}

\begin{lemma}\label{nullcubictermLin}
Let $  t\in [2^{m-1},2^m], m\in \Z_+$. For any $\mu, \nu, \iota\in \{+,-\},$ and bilinear operator $T(\cdot, \cdot)$ with symbol $m\in \mathcal{S}^\infty_{k,k_1,k_2} $, s.t., $\|m\|_{\mathcal{S}^\infty_{k,k_1,k_2}}\lesssim 2^{-k_{1,-}+ 2\max\{k_1,k_2\}_{+}}$,   under the bootstrap assumption \eqref{bootstrap}, we have 
\be\label{2024march4eqn81}
\sup_{k\in \Z} \sum_{\begin{subarray}{c}
 (k_1,k_2)\in  \cup_{i=1,2,3}\chi_{k}^i\\
 (k_1',k_2')\in \cup_{i=1,2,3}\chi_{k_1}^i\\
\end{subarray}} 2^{k/3+10k_{+}}\|\mathcal{F}\big[T\big(P_{k_1}\big( \widetilde{Q}(P_{k_1'}U^{\mu}(t), P_{k_2'}U^{\nu}(t))\big), P_{k_2}U^{\iota}(t)\big)\big](\xi) \|_{L^\infty_\xi}\lesssim 2^{-1.04m}\epsilon_1^3.
\ee
where the bilinear operator $\widetilde{Q}(\cdot, \cdot)$ is defined by the symbol $\widetilde{q}(\xi-\eta, \eta)$  in \eqref{symbolphi}. 
\end{lemma}
\begin{proof}
Recall the symbol \eqref{symbolphi}. By using the volume of support of frequency variable and the $H^{N_0}$-norm of the wave part, we have
\be\label{2024march4eqn82}
\begin{split}
& \sum_{\begin{subarray}{c}
(k_1,k_2)\in \Z^2/[-3m/4, 2N_0/m]^2\\ 
(k_1',k_2')\in \Z^2\\
\end{subarray}}2^{k/3+10k_+}\| \mathcal{F}\big[T\big(P_{k_1}\big( \widetilde{Q}(P_{k_1'}U^{\mu}(t), P_{k_2'}U^{\nu}(t))\big), P_{k_2}U^{\iota}(t)\big)\big](\xi)\|_{L^\infty_\xi}\\ 
&\lesssim \sum_{(k_1,k_2)\in \Z^2/[-3m/4, 2N_0/m]^2}\min\{2^{5\min\{k_1,k_2\}/3}, 2^{-(N-15)\max\{k_1,k_2\}_{+}}\}\epsilon_1^3\lesssim  2^{-5m/4}\epsilon_1^3. 
\end{split}
\ee

Moreover, for any fixed $(k_1, k_2)\in [-3m/4, 2N_0/m]^2$, we can rule out the very small $k_1',k_2'$ case and the relatively big $k_1',k_2'$ case as follows, 
\be\label{2024march4eqn84}
\begin{split}
& \sum_{(k_1',k_2')\in \Z^2/[-3m/4, 2N_0/m]^2}2^{k/3+10k_+}\| \mathcal{F}\big[T\big(P_{k_1}\big( \widetilde{Q}(P_{k_1'}U^{\mu}(t), P_{k_2'}U^{\nu}(t))\big), P_{k_2}U^{\iota}(t)\big)\big](\xi)\|_{L^\infty_\xi}\\ 
&\lesssim \sum_{(k_1,k_2)\in \Z^2/[-3m/4, 2N_0/m]^2}\min\{2^{5\min\{k_1,k_2\}/3}, 2^{-(N-15)\max\{k_1,k_2\}_{+}}\}\epsilon_1^3\lesssim  2^{-5m/4}\epsilon_1^3. 
\end{split}
\ee

  Since there are at most $m^4$ cases left, which only cause logarithmic loss, it would be sufficient to let  $k_1,k_2,k_1',k_2'\in [-3m/4, 2N_0/m]\cap \Z$ be fixed. As in the proof of the obtained estimate \eqref{2024march4eqn91}, we also exploit the null structure of the bilinear operator $\widetilde{Q}(\cdot, \cdot)$. Motivated from the proof of  \eqref{2024march4eqn91}, we decompose into two parts as follows, 
 \be\label{2024march4eqn85}
  \widetilde{Q}(P_{k_1'}U^{\mu}(t), P_{k_2'}U^{\nu}(t))\big)=  \widetilde{Q}^{\leq }(P_{k_1'}U^{\mu}(t), P_{k_2'}U^{\nu}(t))\big)+  \widetilde{Q}^{>}(P_{k_1'}U^{\mu}(t), P_{k_2'}U^{\nu}(t))\big),
 \ee
 where $
\forall \star\in\{\leq , >\} $, 
 \be
 \begin{split}
 \widetilde{Q}^{\star }(P_{k_1'}U^{\mu}(t), P_{k_2'}U^{\nu}(t))\big) &= \int_{\R^3} \int_{\R^3} e^{i x \cdot \xi } \widetilde{q}(\xi-\eta, \eta)\psi_{\star \bar{a}}(\frac{\xi}{|\xi|}\times\frac{\eta}{|\eta|})   \widehat{U^\mu}(t, \xi-\eta)  \widehat{U^\nu}(t, \eta)d\eta d \xi. 
 \end{split}
  \ee
  where $\bar{a}:= -0.45 m-\beta(k_1,k_1',k_2')/2,$ and $\beta(k_1,k_1',k_2')$ is defined in \eqref{2024march3eqn19}.

Since the threshold we choose is $\bar{a}$ instead of $-m/2-\min\{k_1',k_2',k_1\}/2$, and the symbol $\widetilde{q}(\xi-\eta, \eta)$ of the bilinear operator $\widetilde{Q}^{  }(\cdot, \cdot)$ is better than the symbol of the bilinear operator $A_{\mu, \nu}(U^{\mu}(t), U^{\nu}(t))$ in the sense that it provides the smallness of the output frequency, which plays the small role of $2^k$ in \eqref{2024march4eqn91}.     After rerunning the argument used for obtaining \eqref{2024march4eqn91}, we have 
\be\label{2024march5eqn11}
\begin{split}
  &\|P_{k_1}\widetilde{Q}^{  }(P_{k_1'}U^{\mu}(t), P_{k_2'}U^{\nu}(t))\big)\|_{L^2}\lesssim 2^{-m}\epsilon_1^2,\\
 &\|P_{k_1}\widetilde{Q}^{> }(P_{k_1'}U^{\mu}(t), P_{k_2'}U^{\nu}(t))\big)\|_{L^2}\lesssim 2^{-1.1 m}\epsilon_1^2.\\
  \end{split}
\ee
From the first estimate in \eqref{2024march5eqn11}, we can rule out very low output frequency as follows, 
\[
\sup_{k\in (-\infty, -0.15m)\cap Z }2^{k/3+10k_{+}}  \| \mathcal{F}\big[T\big(P_{k_1}\big( \widetilde{Q}(P_{k_1'}U^{\mu}(t), P_{k_2'}U^{\nu}(t))\big), P_{k_2}U^{\iota}(t)\big)\big](\xi)\|_{L^\infty_\xi}\lesssim 2^{-1.05 m}\epsilon_1^3.
\]

  Now, we focus on the case $k\in [-0.15m, 2m/N_0]\cap Z$ and  $k_1,k_2,k_1',k_2'\in [-3m/4, 2N_0/m]\cap \Z$. From the second estimate in \eqref{2024march5eqn11} and the $L^2-L^2$-type bilinear estimate, we rule out further the large angle case as follows,  
\be\label{2024march6eqn21}
2^{k/3+10k_{+}} \|\mathcal{F}\big[T\big(P_{k_1}\big( \widetilde{Q}^{>}(P_{k_1'}U^{\mu}(t), P_{k_2'}U^{\nu}(t))\big), P_{k_2}U^{\iota}(t)\big)\big](\xi)\|_{L^\infty_\xi}\lesssim 2^{-1.05 m}\epsilon_1^3. 
\ee

It remains to estimate the contribution of $\widetilde{Q}^{\leq}(P_{k_1'}U^{\mu}(t), P_{k_2'}U^{\nu}(t))\big)$. Note that,  on the Fourier side, after doing dyadic decomposition for the angle between $\xi-\eta$ and $\pm(\eta-\sigma)$, we have
\[
 \mathcal{F}\big[T\big(P_{k_1}\big( \widetilde{Q}^{\leq }(P_{k_1'}U^{\mu}(t), P_{k_2'}U^{\nu}(t))\big), P_{k_2}U^{\iota}(t)\big)\big](\xi)=\sum_{l_1\in [\bar{l},2]\cap \Z } \mathcal{H}_{k;k_1,k_2;k_1',k_2'}^{\mu,\nu, \iota;l}(t, \xi),  
\]
where
\be\label{2024march6eqn22}
\begin{split}
\mathcal{H}_{k;k_1,k_2;k_1',k_2'}^{\mu,\nu, \iota;l}(t, \xi)&=\int_{\R^3}\int_{\R^3}e^{-it \big( \mu|\eta-\sigma|+\nu|\sigma|+ \iota|\xi-\eta|\big) } \widehat{V^{\mu}_{k_1'}}(t, \eta-\sigma)   \widehat{V^{\nu}_{k_2'}}(t, \sigma)  \widehat{V^{\iota}_{k_2}}(t, \xi-\eta)   \\
& \quad  \times  \widetilde{q}(\eta-\sigma, \sigma) m (\eta, \xi-\eta)\psi_{k_1}(\eta)  \varphi_{l;\bar{l}}(\frac{\xi-\eta}{|\xi-\eta|}\times\frac{\eta-\sigma}{|\eta-\sigma|})   \psi_{\leq \bar{a}}(\frac{\eta}{|\eta|}\times\frac{\sigma}{|\sigma|}) d \eta d \sigma, 
\end{split}
\ee
where $\bar{l}:=-m/2-\min\{k_1',k_1,k_2\}/2.$  Note that
\be\label{2024march5eqn1}
\begin{split}
|\frac{\eta-\sigma}{|\eta-\sigma|}\times \frac{\eta}{|\eta|}|\lesssim 2^{k_2'-k_1'+\bar{a}}, &\Longrightarrow  |\frac{\xi-\eta}{|\xi-\eta|}\times \frac{\eta}{|\eta|}|\lesssim  2^{k_2'-k_1'+\bar{a}} + 2^{l},\\ 
& \Longrightarrow \big|\frac{\xi}{|\xi|}\times \frac{\eta}{|\eta|}\big|\lesssim 2^{k_1-k}\big( 2^{k_2'-k_1'+\bar{a}} + 2^{l}\big). \\
\end{split}
\ee

For the threshold case, i.e., $l=\bar{l}$, we use the volume of support of $\eta, \sigma$. As a result, from the above estimate \eqref{2024march5eqn1}, we have
\be\label{2024march5eqn21}
\begin{split}
&\big|\mathcal{H}_{k;k_1,k_2;k_1',k_2'}^{\mu,\nu, \iota;\bar{l}}(t, \xi)\big|\lesssim  2^{k_1-k}\big( 2^{k_2'-k_1'+\bar{a}} + 2^{l}\big) 2^{2\bar{a}} 2^{4\min\{k_1,k_2\}/3+5\min\{k_1',k_2'\}/3-10\max\{k_1',k_2'\}_{+}-10k_{2,+}}\epsilon_1^3,\\
&\Longrightarrow \quad 2^{k/3+10k_{+}}\big|\mathcal{H}_{k;k_1,k_2;k_1',k_2'}^{\mu,\nu, \iota;\bar{l}}(t, \xi)\big|\lesssim 2^{-1.1 m}\epsilon_1^3. 
\end{split}
\ee
For the threshold case, i.e., $l> \bar{l}$, we do integration by parts in $\eta$ once. As a result, we have
\be\label{2024march6eqn24}
\mathcal{H}_{k;k_1,k_2;k_1',k_2'}^{\mu,\nu, \iota;l}(t, \xi)=\mathcal{H}_{k;k_1,k_2;k_1',k_2'}^{\mu,\nu, \iota;l;1}(t, \xi)+ \mathcal{H}_{k;k_1,k_2;k_1',k_2'}^{\mu,\nu, \iota;l;2}(t, \xi),
\ee
where
\[
\begin{split}
\mathcal{H}_{k;k_1,k_2;k_1',k_2'}^{\mu,\nu, \iota;l;1}(t, \xi)&=\int_{\R^3}\int_{\R^3}e^{-it \big( \mu|\eta-\sigma|+\nu|\sigma|+ \iota|\xi-\eta|\big) }\nabla_\eta \cdot \Big[  \widehat{V^{\mu}_{k_1'}}(t, \eta-\sigma)   \widehat{V^{\iota}_{k_2}}(t, \xi-\eta)  \widehat{V^{\nu}_{k_2'}}(t, \sigma) \Big] \widetilde{q}(\eta-\sigma, \sigma) \\
&  \quad  \times  m (\eta, \xi-\eta) \frac{- i \big( \mu \frac{\eta-\sigma}{|\eta-\sigma|} - \iota \frac{\xi-\eta}{|\xi-\eta|} \big)}{t\big|\mu \frac{\eta-\sigma}{|\eta-\sigma|} - \iota \frac{\xi-\eta}{|\xi-\eta|} \big|^2}  \psi_{k_1}(\eta)  \varphi_{l;\bar{l}}(\frac{\xi-\eta}{|\xi-\eta|}\times\frac{\eta-\sigma}{|\eta-\sigma|})   \psi_{\leq \bar{a}}(\frac{\eta}{|\eta|}\times\frac{\sigma}{|\sigma|})  d \eta d \sigma,\\
\mathcal{H}_{k;k_1,k_2;k_1',k_2'}^{\mu,\nu, \iota;l;2}(t, \xi)&=\int_{\R^3}\int_{\R^3}e^{-it \big( \mu|\eta-\sigma|+\nu|\sigma|+ \iota|\xi-\eta|\big) } \widehat{V^{\mu}_{k_1'}}(t, \eta-\sigma)   \widehat{V^{\iota}_{k_2}}(t, \xi-\eta)     \widehat{V^{\nu}_{k_2'}}(t, \sigma)   \\
& \quad   \times \nabla_\eta \cdot \Big[\frac{-i \big( \mu \frac{\eta-\sigma}{|\eta-\sigma|} - \iota \frac{\xi-\eta}{|\xi-\eta|} \big)}{t\big|\mu \frac{\eta-\sigma}{|\eta-\sigma|} - \iota \frac{\xi-\eta}{|\xi-\eta|} \big|^2}   \varphi_{l;\bar{l}}(\frac{\xi-\eta}{|\xi-\eta|}\times\frac{\eta-\sigma}{|\eta-\sigma|})   \psi_{\leq \bar{a}}(\frac{\eta}{|\eta|}\times\frac{\sigma}{|\sigma|}) \\
&\quad \times   \widetilde{q}(\eta-\sigma, \sigma) m (\eta, \xi-\eta)\psi_{k_1}(\eta)\Big] d \eta d \sigma.\\
\end{split} 
\]
For $\mathcal{H}_{k;k_1,k_2;k_1',k_2'}^{\mu,\nu, \iota;l;1}(t, \xi)$,  we use the $L^2-L^2$ type bilinear estimate and use the volume of support of $\eta, \sigma$. As  a result,  from the estimate \eqref{2024march5eqn1},  we have 
\be\label{2024march5eqn51}
\begin{split}
&2^{k/3+10k_{+}} \big|\mathcal{H}_{k;k_1,k_2;k_1',k_2'}^{\mu,\nu, \iota;l;1}(t, \xi)\big|\\
& \lesssim 2^{k/3+10k_{+}} 2^{-m-l+\bar{a}}\big[ 2^{\min\{k_1, k_2\}-(1-\alpha)k_{2}} 2^{(k_1-k)/2}\big( 2^{k_2'-k_1'+\bar{a}} + 2^{l}\big)^{1/2}\\ 
&\quad \times  2^{4\min\{k_1',k_2'\}/3+\bar{a}-10\min\{k_1',k_2'\}_{+}}      + 2^{4\min\{k_1,k_2\}/3}2^{k_1-k}\big( 2^{k_2'-k_1'+\bar{a}} + 2^{l}\big) 2^{\bar{a}/2} 2^{  \min\{k_1',k_2'\}-\alpha k_1'} \big] \epsilon_1^3\\
& \lesssim 2^{-2k/3+(1+\alpha)k_1+10k_+}\big(2^{-m-l} 2^{5\bar{a}/2+4\min\{k_1',k_2'\}/3-k_1'/2} +2^{-m-l/2} 2^{2\bar{a}+4\min\{k_1',k_2'\}/3} \big) \epsilon_1^3 \\ 
&\quad + 2^{-2k/3+7k_1/3+10k_+} \big(2^{-m+3\bar{a}/2+ \min\{k_1',k_2'\}-\alpha k_1'} + 2^{-m-l+5\bar{a}/2 +k_2'+\min\{k_1',k_2'\}   -(1+\alpha)k_1'} \big)  \epsilon_1^3 \\
&\lesssim 2^{-1.3m}\epsilon_1^3.
\end{split}
\ee

For $\mathcal{H}_{k;k_1,k_2;k_1',k_2'}^{\mu,\nu, \iota;l;2}(t, \xi)$,  we use the volume of support of $\eta, \sigma$.  As a result,  from the estimate \eqref{2024march5eqn1},  we have 
\be\label{2024march6eqn1}
\begin{split}
&2^{k/3+10k_{+}} \big|\mathcal{H}_{k;k_1,k_2;k_1',k_2'}^{\mu,\nu, \iota;l;2}(t, \xi)\big|\\
& \lesssim 2^{k/3+10k_{+}} 2^{-m-l+2\bar{a}}  2^{k_1-k}\big( 2^{k_2'-k_1'+\bar{a}} + 2^{l}\big)\\ 
&\quad \times 2^{2\min\{k_1,k_2\}+2\min\{k_1',k_2'\}-(k_2+k_1'+k_2')/3}\big[2^{-l-\min\{k_2,k_1'\}} + 2^{-\bar{a}-k_1}  \big]\\ 
&\lesssim 2^{-m-2k/3+5\min\{k_1,k_2\}/3+4\min\{k_1',k_2'\}/3+10k_+}\big[2^{-2l-\min\{k_2,k_1'\}+3\bar{a}+k_1+k_2'-k_1'} + 2^{-l+2\bar{a}+k_2'-k_1'}\\ 
&\quad +2^{-l+2\bar{a}+k_1-\min\{k_2, k_1'\}} + 2^{\bar{a}}\big]\lesssim 2^{-1.2m}\epsilon_1^3.  
\end{split}
\ee
Recall the decompositions in \eqref{2024march4eqn85},  \eqref{2024march6eqn22}, and \eqref{2024march6eqn24}. After combining the above two estimates and the obtained estimates \eqref{2024march6eqn21} and \eqref{2024march5eqn21}, we have
\[
\sup_{k\in [ -0.15m, 2m/N_0]\cap Z }2^{k/3+10k_{+}}  \| \mathcal{F}\big[T\big(P_{k_1}\big( \widetilde{Q}(P_{k_1'}U^{\mu}(t), P_{k_2'}U^{\nu}(t))\big), P_{k_2}U^{\iota}(t)\big)\big](\xi)\|_{L^\infty_\xi}\lesssim 2^{-1.05 m}\epsilon_1^3.
\]
Hence finishing the proof of our desired estimate \eqref{2024march4eqn81}. 
\end{proof}   

\subsection{Proof of the theorem \ref{maintheorem}}

In this subsection, we summarize the results we obtained so far and prove our main theorem \ref{maintheorem}. From the obtained estimate \eqref{improvedenergyvel} in Lemma \ref{highenergyu}, the obtained estimate \eqref{improvedenergyvelvec} in Lemma \ref{energyvelovec}, the obtained estimate \eqref{improvedenergywave} in Lemma \ref{energywavepart}, the obtained estimate \eqref{2024march1eqn1} in Lemma \ref{decayvelinf}, the obtained estimate \eqref{2024march1eqn41} in Lemma \ref{Linfwavepart}, the following improved estimate holds in the time interval $[0, T],$
\be\label{bootstrapimproved}
 \begin{split}
\sup_{t\in [0, T]} &\langle t \rangle^{1/2-\delta }\|u(t)\|_{H^{N_0+1}}+ \langle t \rangle^{1/2-2\delta }\big(\sum_{\Gamma\in \{S, \Omega\}} \|\Gamma u(t)\|_{H^{1}} \big)  + \langle t\rangle^{}  Z_u(t) + \langle  t \rangle^{-\delta} \|U(t)\|_{H^{N_0}}\\
 &+ \langle t \rangle^{-2\delta }\big(\sum_{\Gamma\in \{S, \Omega\}} \|U^\Gamma(t)\|_{L^{2}} + \sum_{k\in \Z_{-}} 2^{-\alpha k} \| P_k U^{\Gamma}(t)\|_{L^2} \big) + Z_{\phi}(t) \lesssim \epsilon_0.
\end{split}
 \ee
 Hence finishing the bootstrap argument, i.e., $T$ can be extended to infinity. Moreover, as a by product of the $Z_\phi(t)$-estimate, from the obtained estimate \eqref{2024march1eqn41} in Lemma \ref{Linfwavepart}, we have
 \be
 \begin{split}
&\sup_{k\in \Z} 2^{k/3+ 10k_{+}} \|\int_{t}^{\infty} e^{i s  |\xi|} \widehat{\mathcal{N}_\phi}(s, \xi) \psi_k(\xi) d  s \|_{L^\infty_\xi} \lesssim\langle t \rangle^{-0.04+10\delta}\epsilon_0, \\
 \Longrightarrow \quad & \sup_{k\in \Z} 2^{k/3+ 10k_{+}} \|\big[\widehat{U}(t,\xi)- e^{-it |\xi|} V_\infty(\xi) \big]  \psi_k(\xi) d  s \|_{L^\infty_\xi} \lesssim\langle t \rangle^{-0.04+10\delta}\epsilon_0,\\
\end{split}
 \ee
 where
\[
 V_\infty(\xi) := \widehat{U}_0(\xi) +  \int_{0}^{\infty} e^{i s  |\xi|} \widehat{\mathcal{N}_\phi}(s, \xi)   d  s.
\]
 Hence,  the nonlinear solution $U(t)$ scatters to a linear solution.

\end{document}